\documentclass[a4paper]{article}
\usepackage[latin1]{inputenc}
\usepackage[english,french]{babel}
\usepackage[latin1]{inputenc}
\usepackage[T1]{fontenc}
\usepackage{graphicx}
\usepackage{amssymb}
\usepackage{amsthm}
\usepackage{amsmath,amssymb}
\usepackage[all]{xy}
\usepackage{enumitem}

\title{ D ThÃ©orÃ¨me  D d'ouvert de bifurcations pour des automorphismes de HÃ©non 
 eT intersections d'ensembles de Cantor}
\author{SÃ©bastien Biebler}
\date{Decembre 2013}
\AtEndDocument{    \footnotesize \textsc{LAMA, UMR8050, UNIVERSITE PARIS-EST MARNE LA VALLEE, 5 BOULEVARD DESCARTES, 77454 CHAMPS SUR MARNE, FRANCE} \par \textit{E-mail address} : \texttt{sebastien.biebler@u-pem.fr}}
\usepackage[nottoc, notlof, notlot]{tocbibind}
\begin{document}
 \newtheorem* {myTheo2} {Theorem}  \newtheorem* {myTheo3} {Theorem}
 \newtheorem {df5} [subsubsection]{Definition}  \newtheorem {nt5} [subsubsection] {Notation} \newtheorem {co5}[subsubsection]{Corollary} \newtheorem* {rk}{Remark}  \newtheorem* {coo}{Corollary} 
\newtheorem{df3}[subsubsection]{Definition}
\newtheorem{pr3}[subsubsection]{Proposition}
\newtheorem{lem3}[subsubsection]{Lemma}

\begin{center} \Huge{Latt\`es maps and the interior of the bifurcation locus}
\end{center} \begin{center}  \huge{S\'ebastien Biebler } \end{center}
\selectlanguage{english}
\begin{abstract}   
We study the phenomenon of robust bifurcations in the space of holomorphic maps of $\mathbb{P}^2(\mathbb{C})$. We prove that any Latt\`es example of sufficiently high degree belongs to the closure of the interior of the bifurcation locus. In particular, every Latt\`es map has an iterate with this property. To show this, we design a method creating robust intersections between the limit set of a particular type of iterated functions system in $\mathbb{C}^2$ with a well-oriented complex curve. Then we show that any Latt\`es map of sufficiently high degree can be perturbed so that the perturbed map exhibits this geometry.

 \end{abstract}
\tableofcontents 
\section{Introduction} \small \subsection{Context} 
In the article \cite{mss}, Ma$\tilde{\text{n}}$\'e, Sad and Sullivan, and independently Lyubich in \cite{lyy}, introduced a relevant notion of stability for holomorphic families $(f_{\lambda})_{\lambda \in \Lambda}$ of rational mappings of degree $d$ on the Riemann sphere $\mathbb{P}^{1}(\mathbb{C})$, parameterized by a complex manifold $\Lambda$. The family $(f_{\lambda})_{\lambda \in \Lambda}$ is $J$-stable in a connected open subset $\Omega \subset \Lambda$ if in $\Omega$ the dynamics is structurally stable on the Julia set $J$. It can be shown that this is equivalent to the fact that periodic repelling points stay repelling points inside the given family. The bifurcation set is the complementary of the locus of stability. A remarkable fact is that the $J$-stability locus is dense in $\Lambda$ for every such family. Moreover, parameters with preperiodic critical points are dense in the bifurcation locus. \medskip

 In higher dimension, less is known. We will only discuss the 2-dimensional case in this paper. The research in this field mostly takes inspiration from two different types of maps with different behaviour : polynomial automorphisms of $\mathbb{C}^{2}$ and holomorphic endomorphisms of $\mathbb{P}^{2}(\mathbb{C})$. Knowledge about bifurcations of polynomial automorphisms is growing quickly. Let us quote the work of Dujardin and Lyubich (\cite{dl}) which introduces a satisfactory notion of stability and shows that homoclinic tangencies, which are the 2-dimensional counterpart of preperiodic critical points, are dense in the bifurcation locus. \medskip
  
From now on, we are interested in the case of holomorphic endomorphisms of $\mathbb{P}^{2}(\mathbb{C})$. The natural generalization of the one-dimensional theory was designed by Berteloot, Bianchi and Dupont in \cite{du9}. Their notion of stability is as follows : let $(f_{\lambda})_{\lambda \in \Lambda}$ be a holomorphic family of holomorphic maps of degree $d$ on $\mathbb{P}^{2}(\mathbb{C})$ where $\Lambda$ is simply connected. Then the following assertions are equivalent: \begin{enumerate} 
               \item  The function on $\Lambda$ defined by the sum of Lyapunov exponents of the equilibrium measure $\mu_{f_{ \lambda}} : \lambda \mapsto  \chi_{1}(\lambda) + \cdot \cdot \cdot + \chi_{k} (\lambda)$ is pluriharmonic on $\Lambda$. \item 
 The sets $(J^{*}(f_{\lambda}))_{\lambda \in \Lambda}$ move holomorphically in a weak sense, where $J^{*}(f_{\lambda})$ is the support of the measure $\mu_{f_{ \lambda}}$. \item There is no (classical) Misiurewicz bifurcation in $\Lambda$.
             \item Repelling periodic points contained in $J^{*}(f_{\lambda})$ move holomorphically over $\Lambda$.   \end{enumerate}
If these conditions are satisfied, we say that  $(f_{\lambda})_{\lambda \in \Lambda}$ is $J^{*}$-stable. If $(f_{\lambda})_{\lambda \in \Lambda}$ is not $J^{*}$-stable at a parameter $\lambda_{0}$, we will say that a bifurcation occurs at $\lambda_{0}$. \medskip 

A major difference with the one-dimensional case is the existence of open sets of bifurcations. Recently, several works have shown the existence of persistent bifurcations near well-chosen maps. By \cite{du9}, to obtain open subsets in the bifurcation locus, it is enough to create a persistent intersection between the postcritical set and a hyperbolic repeller contained in $J^{*}$. Dujardin gives in \cite{dm} two mechanisms leading to such persistent intersections. The first one is based on topological considerations and the second uses the notion of blender, which is a hyperbolic set with very special fractal properties. Both enable to get persistent bifurcations near maps of the form $(z,w) \mapsto (p(z),w^{d}+\kappa)$. The results of Dujardin have been improved by Taflin in \cite{taflin}. Taflin shows that if $p$ and $q$ are two polynomials of degree bounded by $d$ such that $p$ is a polynomial corresponding to a bifurcation in the space of polynomials of degree $d$, then the map $(p, q)$ can be approximated by polynomial skew products having an iterate with a blender and then by open sets of bifurcations. Note that the idea of blender arised in the work of Bonatti and Diaz on real diffeomorphisms (\cite{bd1}) and already appeared in holomorphic dynamics in the work of the author (\cite{bieblerberthollove}). \medskip

 Latt\`es maps are holomorphic endomorphisms of $\mathbb{P}^{2}(\mathbb{C})$ which are semi-conjugate to an affine map on some complex torus $\mathbb{T}$ (see \cite{sj} for a classification and \cite{duber} for a characterisation of Latt\` es maps in terms of the maximal entropy measure). It is natural to be interested in these maps in the context of bifurcation theory because their Julia set is equal to the whole projective space $\mathbb{P}^{2}(\mathbb{C})$. This property seems to have a great potential to create persistent intersection between the postcritical set and the Julia set even after perturbation. Berteloot and Bianchi proved in \cite{du3} that the Hausdorff dimension of the bifurcation locus near a Latt\`es map is equal to that of the parameter space. 
 
 \subsection{Main result}
 
  Dujardin asked in \cite{dm} if it was possible to find open sets of bifurcations near any Latt\`es map. In this article we give a partial answer to this question. Here is our main result : 
\begin{myTheo3}
For every two-dimensional complex torus $\mathbb{T}$, there is an integer $d$ (depending on the torus $\mathbb{T}$) such that every Latt\`es map defined on $\mathbb{P}^{2}(  \mathbb{C}) $ of degree $d'>d$ induced by an affine map on $\mathbb{T}$ is in the closure of the interior of the bifurcation locus in $\mathrm{Hol}_{d'}$. \end{myTheo3}
Let us point out the scarcity of tori which are associated to some Latt\`es example on $\mathbb{P}^{2}(  \mathbb{C})$ (the classification is discussed in section 3). We also remark that the degree $d$ is unknown (the situation here is similar to Buzzard's article \cite{bb1}). Moreover, $d$ depends on the torus $\mathbb{T}$ (see subsection 1.3). This is due to the necessity of making only holomorphic perturbations. As a consequence of the theorem we get : \begin{coo}  For every Latt\`es map $L$ of degree $d$, there is an integer $n(L)$ such that for every $n \ge n(L)$, the iterate $L^{n}$ is in the closure of the interior of the bifurcation locus in $\mathrm{Hol}_{d^{n}}$.  \end{coo}
The Theorem also implies that there are no open subsets of Latt\`es maps in the family of endomorphisms of $\mathbb{P}^{2}(\mathbb{C})$ (if one does not need to iterate). Indeed, for such an open set of Latt\`es maps, the Lyapunov exponents would be minimal (see \cite{duber}) and the sum of Lyapunov exponents would be pluriharmonic, but the Theorem implies that this set intersects open sets of bifurcations where the sum of Lyapunov exponents is not pluriharmonic (by \cite{du9}). 

 \subsection{Outline of proof}

 To prove this result, we create persistent intersections between the postcritical set and a hyperbolic repeller contained in the Julia set. Our proof has two main parts : first, we create a toy-model which allows to obtain intersections between the limit set of some particular type of IFS, called correcting IFS, and a quasi-line that is "well-oriented". Then, in a second time, we perturb the Latt\`es map to create both the correcting IFS and the well-oriented curve inside the postcritical set. This construction exhibits properties somehow similar to the \emph{blenders} of Bonatti-Diaz (\cite{bd1}), with the difference that the covering property holds at the level of the tangent maps of the IFS (see also the notion of  \emph{parablenders} appeared in the work of Berger (\cite{be})).

 \medskip
\medskip

In a first part, we develop an intersection principle (see Proposition 2.1.6). A grid of balls $G$ in $\mathbb{C}^{2}$ is the union of a finite number of balls regularly located at $N^{4}$ vertices of a lattice defined by a $\mathbb{R}$-basis of $\mathbb{C}^{2}$. If we consider a line $\mathfrak{C}$, a pigeonhole argument ensures that if $\mathfrak{C}$ is well oriented and $G$ has a sufficient number of balls $N=N(r)$ (where $r$ is the relative size of a ball compared to the mesh of the grid) then $\mathfrak{C}$ intersects a ball of $G$. We consider a class of IFS such that each inverse branch is very close to a homothety. When we iterate them, a drift can appear : the iterates become less and less conformal. Our class of IFS (called correcting IFS) is designed so that they have the property of correcting themselves from the drift. A linear correction principle is given in Proposition 2.2.2. In subsections 2.3 and 2.4, we treat the case of a curve close to a line and an IFS close to be linear. Our interest in such IFS is that any well-oriented quasi-line $\mathfrak{C}$ intersects the limit set of a correcting IFS. To prove this result, which is Proposition 2.4.1, we ensure that at each step the quasi-line $\mathfrak{C}$ intersects a grid of ball $G^{j}$ which is dynamically defined with the inverse branches of the IFS. Then we use inductively the intersection and the correction principles to ensure that at the next step, $\mathfrak{C}$ intersects a grid of balls $G^{j+1}$ with bounded drift. The intersection of the grids $G^{j}$ is in the limit set, so we produce an intersection between $\mathfrak{C}$ and the limit set of the IFS. Since the property of being correcting is open, this intersection is persistent. \medskip

In the second part, we make three successive perturbations of a Latt\`es map $L$, denoted by $L'$, $L''$ and $L'''$, in such a way that $L'''$ has a robust bifurcation. We work in homogenous coordinates and do explicit perturbations of the following form :
$$ [P_{1} : P_{2} : P_{3}] \rightarrow  [P_{1}+R_{1}P_{3} : P_{2}+R_{2}P_{3} : P_{3}]$$
where $R_{1}$ and $R_{2}$ are rational maps. An important technical point (Proposition 3.2.1) is that we can choose the coordinates so that $P_{3}$ splits. Then if $R_{1}$ and $R_{2}$ are well chosen the degree does not change. The first perturbation $L'$ (Propositions 4.4.4 and 4.4.5) is intended to create a correcting IFS in a ball $\mathfrak{B}$ in $\mathbb{C}^{2}$. Another important technical point is that we can find some critical point $c$ which is preperiodic, with associated periodic point $p_{c}$ such that both the preperiod $n_{c}$ and the period $n_{pc}$ of the preperiodic critical orbit are bounded independently of $L$ (see Proposition 3.3.1). Then we want to create a well-oriented quasi-line inside the postcritical set which intersects $\mathfrak{B}$. The second perturbation $L''$ in Lemma 4.5.10 ensures that the postcritical set at $p_{c}$ is not singular. The third and last perturbation $L'''$ is given in Lemma 4.5.11. It is intended to control the differential at $p_{c}$. This allows us to fix the orientation of the postcritical set at $p_{c}$ and then we use the linear dynamics of the Latt\`es map $L$ on the torus $\mathbb{T}$ in order to propagate this geometric property up to $\mathfrak{B}$ (see Proposition 4.5.3). Note that the periodic point need not lie in $\mathfrak{B}$. At this stage we have both a correcting IFS and a well-oriented quasi-line so we are in position to conclude in section 5.
\medskip

In particular, let us point out that the bound $d$ on the degree is fixed in 4.2.12, 4.2.13 and 4.2.14 : $d = \max(d^{1},d^{2},d^{3})$. Here $d^{1}$ is fixed to ensure that there are sufficiently many inverse branches in the IFS to apply Proposition 2.4.1. $d^{2}$ is intended to make the first perturbation possible in Proposition 4.4.4. (section 2 plays an important role in the determination of $d^{2}$). Similarly, $d^{3}$ is fixed to allow the second and third perturbations in Lemmas 4.5.10 and 4.5.11 along the periodic orbit (whose length is bounded in subsection 3.3 and important to fix $d^{3})$. It is also interesting to remark that the bound $d^{2}$ comes from an interpolation. This has some similarities with the article \cite{bb1} where Buzzard uses a Runge approximation with polynomial automorphisms of sufficiently high degree in order to prove the existence of Newhouse phenomenon in the complex setting. In particular, $d^{1}$ depends on the torus (the number of inverse branches depends on the size of a ball $\mathfrak{B}$ depending on $\mathbb{T}$) and it is also the case for $d^{2}$ (which depends on the integer $i(\mathbb{T})$ defined in Proposition 3.2.1). 
\medskip

In section 2, we develop the theory of intersection between a quasi-line and the limit set of a correcting IFS : the intersection principle and the correction principle are respectively stated in subsections 2.1 and 2.2 and we prove the intersection result in subsection 2.4. In section 3, we provide background on Latt\`es maps and prove a few properties which will be useful later. Some complications arise from Latt\`es maps whose linear part is not the identity. In section 4, we develop the perturbative argument. After giving some preliminaries (subsection 4.1) and fixing many constants (subsections 4.2 and 4.3), we create a correcting IFS in subsection 4.4. In subsection 4.5, we create a well oriented curve inside the postcritical set. Finally, we conclude in section 5 by applying the formalism of subsection 2.4 to the perturbed map $L'''$. \medskip

\textbf{Acknowledgments :}
 The author would like to thank his PhD advisor, Romain Dujardin. This research was partially supported by the ANR project LAMBDA, ANR-13-BS01-0002.

\section{Intersecting a curve and the limit set of an IFS}

\subsection{Linear model}
In this section, we will work with an IFS, whose maps are small perturbations of homotheties of the form $\frac{1}{a} \cdot \text{Id}$ with $a \in \mathbb{R}^{*}$ and $|a|>1$. This IFS will be obtained by perturbating a Latt\`es map and its limit set will have persistent intersections with a curve.

\newtheorem{df}[subsubsection]{Definition}
\newtheorem{pr}[subsubsection]{Proposition}
\newtheorem{lem}[subsubsection]{Lemma}

\begin{df} \label{dfp}
Given a $\mathbb{R}$-basis $(u_{1},u_{2},u_{3}, u_{4}) \in (\mathbb{C}^{2})^{4}$, a point $o \in \mathbb{C}^{2}$, an integer $N$ and $r \in (0,1)$, by a grid of balls we mean the union of the balls of radius $r. \min_{1 \le i \le 4} ||u_{i}||$ centered at the points $o+iu_{1}+ju_{2}+ku_{3}+klu_{4}$ where $-N \le i,j,k,l \le N$ are integers. We will denote it by $G = (u,o,N,r)$. The middle part of $G$ is the set $\{o+xu_{1}+yu_{2}+zu_{3}+wu_{4}, 0 \le |x|,|y|,|z|,|w| \le\frac{N}{2}\}$. The hull of $G$ is the set $\{o+xu_{1}+yu_{2}+zu_{3}+wu_{4}, 0 \le |x|,|y|,|z|,|w| \le N\}$. The size of $G$ is $\text{size}(G) =2N \cdot  \max_{1 \le i \le 4} ||u_{i}||$. 

\end{df}
In the following, the parameter $r$ will be bounded from below and we will let $\max_{1 \le i \le 4}  ||u_{i}|| \rightarrow 0$ so that the radius of the balls $r. \text{min}_{1 \le i \le 4}  ||u_{i}||$ will tend to 0. The integer $N$ will be taken sufficiently large to satisfy some conditions  depending on the degree of the Latt\`es map. Herebelow the notions of "opening" and "slope" are relative to the standard euclidean structure of $\mathbb{C}^{2}$.

 \begin{nt5}
 For a non zero vector $w \in \mathbb{C}^{2}$ and $\theta>0$, we will denote $C^{w,\theta}$ the cone of opening $\theta$ centered at $w$. 
 \end{nt5}
  \begin{nt5}
 For any quadruple of non zero vectors $w_{1},w_{2},w_{3},w_{4}$ in $\mathbb{C}^{2}$, we will denote $\overline{w} = \overline{(w_{1},w_{2},w_{3},w_{4})}$ its projection onto $\mathbb{P}(\mathbb{R}^{8})$.  For any matrix $U \in \mathrm{GL}_{2}(\mathbb{C})$, we simply denote by $U \cdot$ the  induced action on $\mathbb{P}(\mathbb{R}^{8})$. 
\end{nt5}

\begin{df} \label{dfpp}
The middle part of a ball (resp. the $\frac{3}{4}$-part) is the ball of same center and $\frac{1}{2}$ times its radius (resp. $\frac{3}{4}$ times its radius).
\end{df}

\begin{df} \label{def33}
A holomorphic curve $C$ is a $(\varepsilon,w)$-quasi-line if $C$ is a graph upon a disk in $\mathbb{C} \cdot w$ of slope bounded by $\varepsilon$ relative to the projection onto $w$. A $(\varepsilon,w)$-quasi-diameter of a ball $\mathbb{B}$ is a $(\varepsilon,w)$-quasi-line $C$ intersecting the ball of same center as $\mathbb{B}$ and of radius $\frac{1}{10}$ times the radius of $\mathbb{B}$.

\end{df}
Here is our "intersection principle" : 

\begin{pr}[Intersection Principle] \label{r42}
 For every $u \in (\mathbb{C}^{2})^{4}$, $r>0$, $\eta>0$ and $w_{0} \in \mathbb{C}^{2}$, there exists a neighborhood $\mathcal{N}(u)$ of $\overline{u}$ in $\mathbb{P}(\mathbb{R}^{8})$, there exists $\theta>0$, $N(r) > 0$ and a vector $w \in \mathbb{C}^{2}$ with $     || w-w_{0} ||<\eta   $ such that the following property (P) holds :
 \newline
 \newline
 (P)  For every grid of balls $G = (v,o,N,r)$ such that $ \overline{v}  \in \mathcal{N}(u)$ and $N > N(r)$, for every $(\theta,w)$-quasi-line of direction in $C^{w,2\theta}$ intersecting the middle part of the grid of balls $G$, there is a non empty intersection between the $(\theta,w)$-quasi-line and the middle part of one of the balls of the grid. \newline \newline Moreover, property (P) stays true for $w'$ sufficiently close to 
 $w$.
  \end{pr}

  \begin{proof} Let us first prove the result in the case of a line intersecting the grid of balls. After composition by a real linear isomorphism if necessary, we can suppose $(u_{1},u_{2},u_{3},u_{4}) = (e_{1},ie_{1},e_{2},ie_{2})$ where $e_{1} = (1,0)$ and $e_{2} = (0,1)$ so that the centers of the balls of the grid have integer coordinates. Let us take $w_{1} = \frac{\alpha_{1}}{\beta}e_{1}+\frac{\alpha_{2}}{\beta}ie_{1}+  \frac{\alpha_{3}}{\beta}e_{2}+\frac{\alpha_{4}}{\beta}ie_{2}$ such that $|| w_{1}-w_{0} ||<\eta  $ with rational coordinates $\alpha_{1},\alpha_{2},\alpha_{3},\alpha_{4}, \beta \in \mathbb{Z}$. We take $m= \lfloor \frac{10}{r} \rfloor$. Then, let us take the vector $w = w_{1}+ \frac{1}{m\beta}e_{1}+\frac{1}{m^{2}\beta}ie_{1}+\frac{1}{m^{3}\beta}e_{2}+\frac{1}{m^{4}\beta}ie_{2}$ and $N > 10\beta m^{5}=  N(r)$. We can increase $m$ if necessary so that $w$ satisfies $     || w-w_{0} ||<\eta   $.

   \begin{lem}There is a non empty intersection between any line of direction in $C^{w,2\theta}$ intersecting the middle part of the grid of balls $G$ and the middle part of one of the balls of the grid of balls if $\theta$ is sufficiently small.\end{lem}

   \begin{proof} We divide each mesh of the lattice into $m^{4}$ hypercubes. To each of these  hypercubes, we can assign the quadruple of integers given by the coordinates of a given corner. Taking new coordinates by making a translation if necessary, we can suppose that the union of the middle parts of the balls of the lattice contains the union of the hypercubes whose four coordinates are all equal to 0 modulo $m$. Let us take a point $x_{0}$ of the line inside the middle part of the lattice, and for every $k \in \mathbb{N}$, we denote : $x_{k} = x_{0}+k \cdot w$. Then, we have that : $$\lfloor x_{k+\beta m,1} \rfloor \equiv  \lfloor x_{k,1} \rfloor +1 \pmod m \text{  and  } \lfloor x_{k+\beta m^{2},2} \rfloor  \equiv \lfloor x_{k,2} \rfloor +1  \pmod m$$ 
  $$\lfloor x_{k+\beta m^{3},3} \rfloor  \equiv \lfloor  x_{k,3} \rfloor +1 \pmod m \text{  and  } \lfloor x_{k+\beta m^{4},4} \rfloor  \equiv \lfloor x_{k,4} \rfloor +1 \pmod m$$
  Since $N >10\beta m^{5} = N(r)$, the previous relations imply there exists some $x_{n}$ which intersects some hypercube of integer coordinates congruent to $(0,0,0,0)$ inside the grid of balls. This implies that the line intersects the middle part of one of the balls of the grid.  \qedhere
  
  \end{proof}

  This intersection persists for any line of direction in $C^{w,2\theta}$ and for any $ \overline{v}$ in a small neighborhood $\mathcal{N}(u)$ of $\overline{u}$. Then, the result stays true if we take $(\theta,w)$-quasi-lines for $\theta$ sufficiently small since property (P) is open for the $C^{1}$ topology and $w'$ sufficiently close to $w$. \qedhere

\end{proof}
The following corollary gives the same conclusion as the previous result but this time with more than one possible direction for the quadruple of vectors of the lattice.
\begin{co5} \label{r55}
 For every finite subgroup $\mathcal{M} \subset \mathrm{Mat}_{2}(\mathbb{C})$, for every $u \in (\mathbb{C}^{2})^{4}$, there exists a neighborhood $\mathcal{N}(u)$ of $\overline{u}$ in $\mathbb{P}(\mathbb{R}^{8})$ such that for every $r>0$, there exists $\theta>0$, $N(r) > 0$ and a vector $w \in \mathbb{C}^{2}$ such that the following property (P) holds :
 \newline
 \newline
 (P)  For every $U \in \mathcal{M}$, for every grid of balls $G= (v,o,N,r)$ with $ \overline{v} \in \mathcal{N}(u) \cup U \cdot  \mathcal{N}(u) \cup  \cdot \cdot \cdot \cup U^{\text{ord}(U)-1} \cdot  \mathcal{N}(u)$ and $N > N(r)$, for every $(\theta,w)$-quasi-line of direction in $C^{w,2\theta}$ intersecting the middle part of $G$, there is a non empty intersection between the $(\theta,w)$-quasi-line and the middle part of one of the balls of the grid. \newline \newline Moreover, this proposition remains true for $w'$ sufficiently close to $w$.

  \end{co5}
\begin{proof}
We just have to apply $\text{ord}(\mathcal{M})$ times Proposition \ref{r42}.
\end{proof}

\subsection{Linear correction principle} 
  
  \begin{nt5}
We will denote by $\mathrm{Mat}_{2}(\mathbb{C})$ the metric space of $(2,2)$ complex matrices with the distance induced by the norm $||.||=||.||_{2,2}$. 
\end{nt5}

In the following, $x$ will be a real positive parameter. We remind that in a first reading it is advised to assume that $U = I_{2}$. The following proposition is the "linear correction principle" we discussed in the introduction.
 
\begin{pr}[Linear correction principle] \label{r45}
 For every finite subgroup $\mathcal{M} \subset \mathrm{Mat}_{2}(\mathbb{C})$, there exists an integer $n>0$, $(n+1)$ balls $V^{0},V^{1},...,V^{n} \subset \mathrm{Mat}_{2}(\mathbb{C})$ such that for every $0<x<1$, there exists a neighborhood $\mathcal{U}_{x} $ of $I_{2}$ in $\mathrm{GL}_{2}(\mathbb{C})$, two open sets $\mathcal{U}'_{x}  \subset \mathcal{U}''_{x} \subset \mathrm{GL}_{2}(\mathbb{C})$ which are union of balls $\mathcal{U}'_{x}  = \bigcup_{1 \le p \le n} (\mathcal{U}'_{x})^{p}$ and $\mathcal{U}''_{x}  = \bigcup_{1 \le p \le n} (\mathcal{U}''_{x})^{p}$ such that : $ (\mathcal{U}'_{x})^{p} \subset  (\mathcal{U}''_{x})^{p}$ for each $1 \le p \le n$ with the following properties : 
\newline  \newline
(i) If $M  \in \mathcal{U}_{x}$, $U \in \mathcal{M}$ and $j \in \mathbb{N}$, then for every $M_{0} \in (x \cdot V^{0}$) : $$ U^{j}MU(I_{2}+M_{0})   \in U^{j+1} \cdot (\mathcal{U}_{x} \cup \mathcal{U}'_{x}) $$
 (ii) If $M \in \mathcal{U}'_{x}$, $U \in \mathcal{M}$ and $j \in \mathbb{N}$, then there exist two integers $1 \le p,p' \le n$ such that $M \in   (\mathcal{U}'_{x})^{p}$ with the property that for every $M_{0} \in (x \cdot  V^{0})$ and for every $M_{p'} \in (x \cdot V^{p'})$, we have : $$ U^{j}MU(I_{2}+M_{0}) \in  U^{j+1} \cdot (\mathcal{U}''_{x})^{p}$$ 
  $$ U^{j}MU(I_{2}+M_{0})U(I_{2}+M_{p'}) \in U^{j+2} \cdot \mathcal{U}_{x} $$

 \end{pr}
\begin{proof}
We consider the vector space $\mathrm{Mat}_{2}(\mathbb{C}) \simeq \mathbb{R}^{8}$. Let us consider a covering of the sphere of center 0 of radius $r$ (which will be chosen later) $S(0,r)$ by $n$ balls $B(X_{i},\frac{1}{20}r)$ of radius $\frac{1}{20}r$. The following geometrical lemma is trivial : 

\begin{lem} \label{huryy} 
For every $1 \le p \le n$, $X \in B(X_{i},\frac{1}{10}r)$, we have : $||X-X_{i}||< \frac{1}{2}r$

\end{lem}

Now, let us call $\mathcal{U}_{1} = B(I_{2},r)$, $(\mathcal{U}'_{1})^{p} = B(I_{2}+X_{i},\frac{1}{20}r)$ and $(\mathcal{U}''_{1})^{p} = B(I_{2}+X_{i},\frac{1}{10}r)$ for each $1 \le p \le n$, $\mathcal{U}'_{1}  = \bigcup_{1 \le p \le n} (\mathcal{U}'_{1})^{p} $ and  $\mathcal{U}''_{1}  = \bigcup_{1 \le p \le n} (\mathcal{U}''_{1})^{p} $. Increasing the number $n$ of open sets $(\mathcal{U}'_{1})^{p}$ if necessary, we can suppose that for every $U \in \mathcal{M}$ and for each $p \le n$, there exists $p' \le n$ such that $(\mathcal{U}'_{1})^{p } \cdot U = U \cdot (\mathcal{U}'_{1})^{p'}$ and $(\mathcal{U}''_{1})^{p} \cdot U = U \cdot (\mathcal{U}''_{1})^{p'}$.

\begin{lem}
There exists $r_{0}>0$ such that if $r<r_{0}$, for every $Y \in (\mathcal{U}''_{1})^{p}  $ : $Y(I_{2}-X_{i}) \in B(I_{2},\frac{1}{2}r)$.
\end{lem}

\begin{proof} The Taylor formula gives us that at 0 at the first order in $X$ : $$(I_{2}+X)(I_{2}-X_{i}) = I_{2}+X-X_{i} +O(r^{2})$$ Then, if $r$ is sufficiently small, Lemma \ref{huryy} implies that for every $X \in B(X_{i},\frac{1}{10}r)$ : $$           (I_{2}+X)(I_{2}-X_{i}) \in B(I_{2},\frac{1}{2}r)$$  
This means that for every $Y \in (\mathcal{U}''_{1})^{p}  $ : $Y(I_{2}-X_{i}) \in B(I_{2},\frac{1}{2}r)$.
 \qedhere

\end{proof}
Now, it is clear it is possible to take sufficiently small balls $V^{0},V^{1},...,V^{n}$ centered at $0,-X_{1},...,-X_{n}$ such that : \newline - If $M  \in \mathcal{U}_{1}$, then for every $M_{0} \in V^{0}$, we have : $ M(I_{2}+M_{0})   \in (\mathcal{U}_{1} \cup \mathcal{U}'_{1}) $
 \newline - If $M \in \mathcal{U}'_{1}$, then there exists $1 \le p \le n$ such that $M \in (\mathcal{U}'_{1})^{p} $ and for every $M_{0} \in V^{0}$, we have : $ M(I_{2}+M_{0}) \in (\mathcal{U}''_{1})^{p} $.\newline \newline The previous lemma implies that if $M \in (\mathcal{U}'_{1})^{p} $ and $M_{0} \in V^{0}$ are such that $ M(I_{2}+M_{0}) \in (\mathcal{U}''_{1})^{p}  $, then for every $M_{p} \in V^{p}$, we have that : $ M(I_{2}+M_{0})(I_{2}+M_{p}) \in    \mathcal{U}_{1} $. Then, properties $(i)$ and $(ii)$ are verified for $x = 1$ and $U = I_{2}$. For each $0<x<1$, let us take the balls $x \cdot V^{0}, x \cdot V^{1},...,x \cdot V^{n} \subset \mathrm{Mat}_{2}(\mathbb{C})$ and let us apply the homothety of factor $x$ of center $I_{2}$ to the sets $\mathcal{U}_{1}$, $\mathcal{U}'_{1}$, $\mathcal{U}''_{1}$, $(\mathcal{U}'_{1})^{p}$ and $(\mathcal{U}''_{1})^{p}$ to get the sets $\mathcal{U}_{x}$, $\mathcal{U}'_{x}$, $\mathcal{U}''_{x}$, $(\mathcal{U}'_{x})^{p}$ and $(\mathcal{U}''_{x})^{p}$ such that properties $(i)$ and $(ii)$ are verified for $x < 1$ and $U = I_{2}$. \newline \newline
 Let us now suppose that $U \neq I_{2}$. The inclusions $ U^{j}MU(I_{2}+M_{0})   \in U^{j+1} \cdot ( \mathcal{U}_{x} \cup \mathcal{U}'_{x}) $ and $ U^{j}MU(I_{2}+M_{0}) \in  U^{j+1} \cdot (\mathcal{U}''_{x})^{p}$ are still true by reducing $V^{0}$ a finite number of times if necessary. Let us take $p \le n$ and $p' \le n$ such that  $(\mathcal{U}'_{1})^{p } \cdot U =U \cdot  (\mathcal{U}'_{1})^{p'}$ and $M_{p'} \in V^{p'}$. Then :   $$ U^{j} \cdot (\mathcal{U}''_{x})^{p} \cdot U(I_{2}+M_{p'}) = U^{j} \cdot U \cdot (\mathcal{U}''_{x})^{p'} \cdot (I_{2}+M_{p'}) $$ $$= U^{j+1} \cdot (\mathcal{U}''_{x})^{p'} \cdot (I_{2}+M_{p'}) \subset U^{j+1} \cdot \mathcal{U}_{x}  $$ This implies that for every $M_{0} \in (x \cdot  V^{0})$ and for every $M_{p'} \in (x \cdot V^{p'})$, we have  $U^{j}MU(I_{2}+M_{0}) \in  U^{j+1} \cdot (\mathcal{U}''_{x})^{p}$ and $ U^{j}MU(I_{2}+M_{0})U(I_{2}+M_{p'}) \in U^{j+2} \cdot \mathcal{U}_{x} $, which concludes the proof of the result.  \qedhere

\end{proof}

Let us point out the following obvious result for later reference. Remind that $\mathcal{N}(u)$ was defined in Proposition \ref{r42} and $\mathcal{U}''_{x}$ comes from Proposition \ref{r45}.

\begin{pr} \label{r53}
 For every $u \in (\mathbb{C}^{2})^{4}$, there exists a number $x(u)>0$ such that for every $0< x< x(u)$, for every $M  \in  ( \mathcal{U}_{x} \cup \mathcal{U}''_{x})$, then $M \cdot \overline{u}$ belongs to $\mathcal{N}(u)$.

\end{pr}
\subsection{Quasi-linear model}
Here we slightly perturb the linear maps we used before but we show we can keep results on persistent intersections. Let us recall that the integer $n$ was defined in Proposition \ref{r45}. Let us remind that  $\mathcal{M} \subset \mathrm{Mat}_{2}(\mathbb{C})$ is a finite subgroup.

\begin{df} \label{type} 
Let $f$ be a linear map defined on an open subset $\mathcal{V}$ of $\mathbb{C}^{2}$. We say that $f$ is linear of type $(x,p)$ for any $0 \le p \le n$ if $f$ can be written $f = \frac{1}{a}(A+h)$ with $a \in \mathbb{C}^{*}$, $A \in \mathcal{M}$ and $h \in x \cdot V^{p}$ (where $V^{p}$ was defined in Proposition \ref{r45}). The modulus $|a|$ is called the contraction factor of $f$.
\newline \newline
Let $f$ be a smooth map defined on an open subset $\mathcal{V}$ of $\mathbb{C}^{2}$. We say that $f$ is quasi-linear of type $(x,p)$ if the differential $Df_{o}$ is linear of type $(x,p)$ for every $o \in \mathcal{V}$, this is : $f = \frac{1}{a}(A+\tilde{h}) $ with $\tilde{h}$ smooth and $D\tilde{h}_{o} \in x \cdot V^{p}$ for every $o \in \mathcal{V}$ ($A$ and $a$ depend only on $f$ but not on $o$).  \end{df}
The following can be seen as a consequence of Proposition \ref{r45} in the quasi-linear setting. Remember that $x(u)>0$ was defined in Proposition \ref{r53}.
\begin{pr} \label{r51}

Let $\mathcal{M} \subset \mathrm{Mat}_{2}(\mathbb{C})$ be a finite subgroup of unitary matrices. Reducing $x(u)$ if necessary, for every grid of balls $G  = (u,o,N,r)$, for every quasi-linear map $f$ of type $(x,p)$ such that $x<x(u)$ with $ 2  |a|  \cdot  \mathrm{size}(G) \cdot ||f||_{C^{2}} < \frac{r}{2}$ and such that $G$ is included in the domain of $f$, there is a grid of balls $G'  = (u',o',N,r')$ included inside $f(G)$ with $\overline{u'} = (Df)_{o} \cdot \overline{u} $ and :  $$r' = r -    2|a|  \cdot  \mathrm{size}(G) \cdot ||f||_{C^{2}}     $$

\end{pr}
\begin{proof}
We just have to take $o' = f(o)$, $\overline{u'} = (Df)_{o} \cdot \overline{u}$ with $Df_{o}$ linear of type $(x,p)$. Remind that by definition, the size of $G$ is $\text{size}(G) =2N \cdot  \max_{1 \le i \le 4} ||u_{i}||$. When $||f||_{C^{2}} = 0$, the image of $G$ under $f$ is a grid of balls $G'$ of same relative size $r' = r$, each ball of $G'$ is the image of a ball of $G$ under $f = \frac{1}{a}(A+h)$. Reducing $x(u)>0$ (independently of $f$) if necessary, we have that the radius of a ball of $G'$ is between $\frac{1}{2|a|}$ and $\frac{2}{|a|}$ times the radius of a ball of $G$. \newline \newline If $||f||_{C^{2}} \neq 0$, the image of each ball of $G$ under $f$ still contains a ball of $G'$ but this time by the Taylor formula there is an additive term smaller than $\text{size}(G) \cdot ||f||_{C^{2}}$ in the differential of $f$. Then the relative size $r'$ is such that : $$r' \ge r - \frac{\text{size}(G) \cdot ||f||_{C^{2}}}{\frac{1}{2|a|}} = r- 2|a| \cdot \text{size}(G) \cdot ||f||_{C^{2}}$$  \qedhere

\end{proof}

\subsection{Intersecting a curve and the limit set of an IFS in $\mathbb{C}^{2}$}

In this subsection, we give an abstract condition ensuring the existence of an intersection (robust by construction) between a holomorphic curve in $\mathbb{C}^{2}$ and the limit set of an IFS. This will be the model for robust bifurcations near Latt\`es maps. Remind that $n$ was defined in Proposition \ref{r45}, $N$, $w,\theta$ in Proposition \ref{r42}. Remind that the middle part and the $\frac{3}{4}$-part of a ball were defined in Definition \ref{dfpp}, the middle part and the hull of a grid of balls were defined in Definition \ref{dfp}. In the following, for a holomorphic map $\mathcal{G}$ defined on a (closed) ball $\mathfrak{B} \subset \mathbb{C}^{2}$, we will denote $|| \mathcal{G} ||_{C^{2}} = \max_{\mathfrak{B}} || D^{2}\mathcal{G} ||$.

\begin{pr} \label{r99999} Let $(\mathcal{G}_{1},...,\mathcal{G}_{q})$ be a IFS given by $q$ maps defined on a ball $\mathfrak{B} \subset \mathbb{C}^{2}$ of radius $R>0$ satisfying the following properties : \begin{enumerate} \item $\bigcup_{ 1 \le j \le q}  \mathcal{G}_{j}  (\mathfrak{B})$ contains a grid of balls $G^{1}= (u^{1},o^{1},n_{G},r^{1})$ with $q = (2n_{G}+1)^{4}$ such that each $\mathcal{G}_{j}  (\mathfrak{B})$ contains a ball of $G^{1}$ \item the contraction factor of the IFS $(\mathcal{G}_{1},...,\mathcal{G}_{q})$ is $|a| \ge 2$ \item  there exist $(n+1)$ balls $\mathfrak{B}_{0},\mathfrak{B}_{1},...,\mathfrak{B}_{n}  \subset \mathfrak{B}$ of radius larger than $\nu \cdot R$ (with $0<\nu<1$), such that the $\frac{3}{4}$-parts of $\mathfrak{B}_{0},\mathfrak{B}_{1},...,\mathfrak{B}_{n}$ are included in the hull of $G^{1}$, and satisfying the following property :  for each $1 \le j \le q$ such that $\mathcal{G}_{j}(\mathfrak{B}) \subset \mathfrak{B}_{p}$, $\mathcal{G}_{j} = \frac{1}{a}(A+\tilde{h}_{j})$ is quasi-linear of type $(x,p)$ with $x<x(u^{1})$ and $a$,$A$ do not depend on $j$. Moreover, $\bigcup_{ 1 \le j \le q}  \mathcal{G}_{j}  (\mathfrak{B}_{p})$ contains a grid of balls $\Gamma^{1}_{p}= (u^{1},o^{1}_{p},n_{G},s^{1})$ for each $0 \le p \le n$ with $s^{1} \ge \nu \cdot r^{1}$
\item  $n_{G}>\frac{10}{\nu} \cdot N(\frac{\nu \cdot r^{1}}{10})$  \item $|a| \cdot R \cdot \max_{1 \le j \le q} (|| \mathcal{G}_{j} ||_{C^{2}}   ) < \frac{\nu \cdot r^{1}}{100}$
 \end{enumerate}  Let $\mathfrak{C}$ be a $(\theta,w)$-quasi-line of direction in $C^{w,2\theta}$ such that $\mathfrak{C}$ intersects the middle part of $G^{1}$.  \newline \newline Then $\mathfrak{C}$ intersects the limit set of the IFS $(\mathcal{G}_{1},...,\mathcal{G}_{q}) $.
\end{pr}

\begin{df}
We say that $(\mathcal{G}_{1},...,\mathcal{G}_{q})$ is a correcting IFS when the conditions 1, 2, 3, 4 and 5 are satisfied. 
\end{df}

Proposition \ref{r99999} will be the immediate consequence of the following lemma : 
 \begin{lem}
There exist $(n+2)$ sequences of grids $(G^{j})_{j \ge 1}  = (u^{j},o^{j},n_{G},r^{j})_{j \ge 1}$ and $(\Gamma^{j}_{p})_{j \ge 1} = (u^{j},o^{j}_{p},n_{G},s^{j} )_{j \ge 1}  $ with  $0 \le p \le n$ such that we have the following properties : \begin{enumerate}
 \item For every $j>1$, $G^{j}$ is included inside a ball of $G^{j-1}$ and for every $j > 1$, there are $i_{1},...,i_{j-1}  \le q$ such that : $G^{j} \subset (\mathcal{G}_{i_{1}} \circ ... \circ \mathcal{G}_{i_{j-1}})(G^{1})$ \item For every $j > 1$, $0 \le p \le n$ : $\Gamma^{j}_{p} \subset (  \mathcal{G}_{i_{1}} \circ ... \circ \mathcal{G}_{i_{j-1}})(\Gamma^{1}_{p})$ \item For every $j \ge1$, $D(\mathcal{G}_{i_{1}} \circ ... \circ \mathcal{G}_{i_{j-1}})_{o^{1}}  \in A^{j-1} \cdot (\mathcal{U}_{x} \cup \mathcal{U}''_{x})$ and for $j \ge 2$ : $$r^{j} \ge r^{1}- 2|a| \cdot R \cdot \max_{1 \le j \le q} (|| \mathcal{G}_{j}  ||_{C^{2}}   )  \sum^{j-2}_{ l \ge 0} \frac{1}{|a|^{l}}     \text{ and } s^{j} \ge s^{1}- 2|a| \cdot  R \cdot  \max_{1 \le j \le q} (|| \mathcal{G}_{j} ||_{C^{2}}   )  \sum_{ l \ge 0}^{j-2} \frac{1}{|a|^{l}}     $$ 
 \item For every $j > 1$, there exists $1 \le p_{j} \le n$ such that the quasi-line $\mathfrak{C}$ intersects the middle part of a ball of $\Gamma^{j}_{p_{j}}$  
 \end{enumerate}
 \end{lem}

\begin{proof} The proof of this lemma is based on an induction procedure. We begin by an initialisation called Case 0 where we pick the grids of balls at the first level $G^{1}$ and $\Gamma^{1}_{p}$ for $ 0 \le p \le n$. We intersect for the first time the quasi-line $\mathfrak{C}$ with a ball and we construct the grids at the second level. Case 0 is somewhat different from the rest of the demonstration because we do not not control the initial position of $\mathfrak{C}$. Then, Case 1 has to be thought as the most frequent situation : $\mathfrak{C}$ intersects a grid of balls whose geometry is good enough, and we can intersect $\mathfrak{C}$ with a new grid whose geometry is very close to the previous one. Then, it may happen a time when the geometry of this grid is too deformed. Then we apply a "correction" (Cases 2 and 3), which leads back to Case 1. 
\newline \newline  $\text{\underline{Case 0} : initialization}$ \newline \newline  By hypothesis, $\bigcup_{ 1 \le j \le q}  \mathcal{G}_{j}  (\mathfrak{B})$ contains a grid of balls $G^{1}= (u^{1},o^{1},n_{G},r^{1})$ and similarly $\bigcup_{ 1 \le j \le q}  \mathcal{G}_{j}  (\mathfrak{B}_{p})$ contains a grid of balls $\Gamma^{1}_{p}= (u^{1},o^{1}_{p},n_{G},s^{1})$ for each $0 \le p \le n$ with $s^{1} \ge \nu \cdot r^{1}$. So, for the first step $j = 1$, the $(n+2)$ grids of balls are already constructed. \newline \newline By Corollary \ref{r55}, $\mathfrak{C}$ intersects in its middle part a ball of $\Gamma^{1}_{0}$ : indeed, $\Gamma^{1}_{0}$ is a grid of balls such that $\overline{ u^{1}} \in \mathcal{N}(u^{1})$, we have $ s^{1} \ge \nu \cdot r^{1} > \frac{\nu \cdot r^{1}}{10}$ and $ n_{G}>\frac{10}{\nu} \cdot N(\frac{\nu \cdot r^{1}}{10}) >N(\frac{\nu \cdot r^{1}}{10})$ (beware that the matrix $A$ corresponds to the matrix denoted by $U$ in Corollary \ref{r55}). Then it intersects the ball $\mathcal{G}_{i_{1}}(\mathfrak{B})$ of $G^{1}$ which contains this ball of $\Gamma^{1}_{0}$. According to Proposition \ref{r51}, there exists a grid of balls $G^{2} =  (u^{2},o^{2},n_{G},r^{2})$ included in $\mathcal{G}_{i_{1}}(G^{1}) $. We have $\overline{u^{2}} \in A \cdot \mathcal{N}(u^{1})$ and $r^{2} \ge r^{1} - 2|a| \cdot \text{size}(G^{1}) \cdot \max_{1 \le j \le q} (|| \mathcal{G}_{j} ||_{C^{2}}   ) \ge r^{1} -  2|a| \cdot R \cdot  \max_{1 \le j \le q} (|| \mathcal{G}_{j} ||_{C^{2}}   )>\frac{\nu \cdot r^{1}}{10} $. Applying  Proposition \ref{r51} to the grids of balls $\Gamma^{1}_{p}$ ($0 \le p \le n$), there exist $(n+1)$ grids of balls $\Gamma^{2}_{p} =  (u^{2},o^{2}_{p},n_{G},s^{2})$ included in $\mathcal{G}_{i_{1}}(\Gamma^{1}_{p}) $ for $0 \le p \le n$. We have : $\overline{u^{2}} \in A \cdot \mathcal{N}(u^{1})$ and $s^{2} \ge s^{1} - 2|a| \cdot \text{size}(\Gamma^{1}_{p}) \cdot \max_{1 \le j \le q} (|| \mathcal{G}_{j} ||_{C^{2}}   ) \ge s^{1} -  2|a| \cdot R \cdot  \max_{1 \le j \le q} (|| \mathcal{G}_{j} ||_{C^{2}}   )>\frac{\nu \cdot r^{1}}{10} $.    \newline \newline  Let us now suppose by induction that the $(n+2)$ sequences of grids of balls satisfying (1),(2),(3) and (4) are constructed up to step $j$ with the additional properties that $\mathfrak{C}$ intersects in its middle part a ball of $\Gamma^{j-1}_{p_{j-1}}$ and that the following property is verified : 
\newline 
\newline 
(Q) For every $i$ such that $(\mathcal{G}_{i_{1}} \circ ... \circ \mathcal{G}_{i_{j-1}} \circ \mathcal{G}_{i})(G^{1}) \subset  \Gamma^{j-1}_{p_{j-1}}$, we have : $$D(\mathcal{G}_{i_{1}} \circ ... \circ \mathcal{G}_{i_{j-1}}  \circ \mathcal{G}_{i})_{o^{1}}  \in A^{j} \cdot (\mathcal{U}_{x} \cup \mathcal{U}''_{x})$$
Let us construct the grids of balls at the next step. The proof is inductive, at each step of the proof we are in one of the three cases we are going to discuss, which differ by two parameters. We have a quasi-line intersecting a grid of balls and we have to make a different choice to intersect a ball corresponding to one of the $(n+1)$ types of differentials we introduced earlier. Note that after Case 0, we will necessarily be in Case 1. \newline \newline
 $  \text{\underline{Case 1} : }   D(\mathcal{G}_{i_{1}} \circ ... \circ \mathcal{G}_{i_{j-1}} )_{o^{1}}    \in A^{j-1} \cdot \mathcal{U}_{x}  \text{   and    }p_{j-1} = 0 $ \newline \newline
By construction, $\mathfrak{C}$ intersects in its middle part a ball $B^{j-1}_{0}$ of the grid of balls $\Gamma^{j-1}_{0}$.   Since $\Gamma^{j}_{0}  = (u^{j},o^{j}_{0},n_{G},s^{j} )$ is a grid of balls such that $ \Gamma^{j}_{0}  \subset (  \mathcal{G}_{i_{1}} \circ ... \circ \mathcal{G}_{i_{j-1}})(\Gamma^{1}_{0})$ and $D(\mathcal{G}_{i_{1}} \circ ... \circ \mathcal{G}_{i_{j-1}} )_{o^{1}}  \in A^{j-1} \cdot \mathcal{U}_{x}$, we have according to Proposition \ref{r53} that $\overline{u^{j}} \in A^{j-1} \cdot \mathcal{N}(u^{1})$. The relative size of $\mathfrak{B}_{0}$ compared to $\mathfrak{B}$ is equal to $\nu$, the $\frac{3}{4}$-part of $\mathfrak{B}_{0}$ is included in the hull of $G^{1}$ and $n_{G}>\frac{10}{\nu} \cdot N(\frac{\nu \cdot r^{1}}{10})$. Then it is possible to take an union of balls of $\Gamma^{j}_{0}$ included in $B^{j-1}_{0}$ which form a grid of balls $\Gamma'$ of basis $u^{j}$, relative size $s^{j}$ and with $(\frac{1}{10}\nu \cdot n_{G})^{4}$ balls. By construction, we can take $\Gamma'$ such that $\mathfrak{C}$ intersects the middle part of $\Gamma'$. Since $\overline{u^{j}} \in A^{j-1} \cdot \mathcal{N}(u^{1})$, $s^{j}>\frac{\nu \cdot r^{1}}{10}$ and $\frac{1}{10}\nu \cdot n_{G} > N(\frac{\nu \cdot r^{1}}{10})$ we have according to Corollary \ref{r55} that $\mathfrak{C}$ intersects in its middle part a ball of $\Gamma^{j}_{0}$. This ball is included inside $ (\mathcal{G}_{i_{1}} \circ ... \circ \mathcal{G}_{i_{j}})(\mathfrak{B})$ with $\mathcal{G}_{i_{j}}$ quasi-linear of type 0. \newline \newline  In particular, $\mathfrak{C}$ intersects the hull of a new grid of balls $G^{j+1} \subset (\mathcal{G}_{i_{1}} \circ ... \circ \mathcal{G}_{i_{j}})(G^{1})$. According to Propositions \ref{r45}, \ref{r51} and Property (Q), $G^{j+1}$ is a grid of balls $G^{j+1}= (u^{j+1},o^{j+1},n_{G} ,r^{j+1})$ with  $D(\mathcal{G}_{i_{1}} \circ ... \circ \mathcal{G}_{i_{j}} )_{o^{1}}     \in A^{j} \cdot (\mathcal{U}_{x} \cup \mathcal{U}'_{x})$ and : 
 $$r^{j+1}  \ge r^{j}- 2|a| \cdot \text{size}(G^{j}) \cdot \max_{1 \le j \le q} (|| \mathcal{G}_{j}  ||_{C^{2}}   ) \ge r^{j}- 2|a| \cdot \frac{R}{|a|^{j-1}} \cdot \max_{1 \le j \le q} (|| \mathcal{G}_{j}  ||_{C^{2}}   ) $$ $$ r^{j+1} \ge r^{1}-   2|a| \cdot   R \cdot \max_{1 \le j \le q} (|| \mathcal{G}_{j}  ||_{C^{2}}   )  \sum_{ l \ge 1}^{j-1} \frac{1}{|a|^{l}}>\frac{\nu \cdot r^{1}}{10}    $$   
Still according to Propositions \ref{r45}, \ref{r51} and Property (Q), there exist $(n+1)$ grids of balls $\Gamma^{j+1}_{p}$ (for $0 \le p \le n$) included in $(\mathcal{G}_{i_{1}} \circ ... \circ \mathcal{G}_{i_{j}})(\Gamma^{1}_{p})$ such that : 
$$s^{j+1}  \ge s^{j}- 2|a|\cdot \text{size}(\Gamma^{j}_{p}) \cdot \max_{1 \le j \le q} (|| \mathcal{G}_{j}  ||_{C^{2}}   ) \ge s^{j}- 2|a| \cdot \frac{R}{|a|^{j-1}} \cdot \max_{1 \le j \le q} (|| \mathcal{G}_{j}  ||_{C^{2}}   ) $$ $$ s^{j+1} \ge s^{1}-   2|a| \cdot   R \cdot \max_{1 \le j \le q} (|| \mathcal{G}_{j}  ||_{C^{2}}   )  \sum_{ l \ge 1}^{j-1} \frac{1}{|a|^{l}}>\frac{\nu \cdot r^{1}}{10}    $$     
 The grids of balls $G^{j+1}$ and  $\Gamma^{j+1}_{p}$ (for $0 \le p \le n$) satisfy (1),(2),(3),(4). In particular, $\mathfrak{C}$ intersects in its middle part a ball of $\Gamma^{j}_{0}$.\newline \newline
 Since $D(\mathcal{G}_{i_{1}} \circ ... \circ \mathcal{G}_{i_{j}} )_{o^{1}}    \in A^{j} \cdot (\mathcal{U}_{x} \cup \mathcal{U}'_{x})$ and $p_{j} = 0$, by Proposition \ref{r45}, if  $(\mathcal{G}_{i_{1}} \circ ... \circ \mathcal{G}_{i_{j}} \circ \mathcal{G}_{i})(G^{1}) \subset  \Gamma^{j}_{p_{j}}$ for some $i$, then $(\mathcal{G}_{i_{1}} \circ ... \circ \mathcal{G}_{i_{j}} \circ \mathcal{G}_{i})(G^{1})$ contains a grid of balls $(u^{j,i},o^{j,i},n_{G},r^{j,i})_{j}$ such that $D(\mathcal{G}_{i_{1}} \circ ... \circ \mathcal{G}_{i_{j}} \circ \mathcal{G}_{i} )_{o^{1}}    \in A^{j+1} \cdot ( \mathcal{U}_{x} \cup  \mathcal{U}''_{x})$, this means that $(Q)$ is verified. \newline \newline Then, after Case 1 and according to Proposition \ref{r45}, only two cases can occur. If $D(\mathcal{G}_{i_{1}} \circ ... \circ \mathcal{G}_{i_{j}} )_{o^{1}}    \in A^{j} \cdot \mathcal{U}_{x}$ and we can apply Case 1 once again. If $D(\mathcal{G}_{i_{1}} \circ ... \circ \mathcal{G}_{i_{j}}  )_{o^{1}} \in A^{j} \cdot (   \mathcal{U}'_{x} - (\mathcal{U}'_{x} \cap  \mathcal{U}_{x})) $, there exists $1 \le p \le n$ such that $D(\mathcal{G}_{i_{1}} \circ ... \circ \mathcal{G}_{i_{j}})_{o^{1}}    \in A^{j} \cdot   (\mathcal{U}')^{p}$. In this case, we are going to "correct" the next grids in a procedure given by Cases 2 and 3. \newline \newline
 $  \text{\underline{Case 2:} }   D(\mathcal{G}_{i_{1}} \circ ... \circ \mathcal{G}_{i_{j-1}} )_{o^{1}}    \in A^{j-1} \cdot (\mathcal{U}'_{x})^{p}  \text{   and    }p_{j-1} = 0 $ \newline
\newline
By construction, $\mathfrak{C}$ intersects in its middle part a ball of the grid of balls $\Gamma^{j-1}_{0}$. We have according to Proposition \ref{r53}  that $\overline{u^{j}} \in A^{j-1} \cdot \mathcal{N}(u^{1})$. Then, using the same argument as in Case 1, we have according to Corollary \ref{r55} that $\mathfrak{C}$ intersects in its middle part a ball of $\Gamma^{j}_{p_{j}}$ included inside $ (\mathcal{G}_{i_{1}} \circ ... \circ \mathcal{G}_{i_{j-1}})(\mathfrak{B})$ where $\mathcal{G}_{i_{j-1}}$ is quasi-linear of type 0 and $p_{j} = p'$ is chosen according to Proposition \ref{r45}. In particular, $\mathfrak{C}$ intersects the hull of a new grid of balls $G^{j+1} \subset (\mathcal{G}_{i_{1}} \circ ... \circ \mathcal{G}_{i_{j}})(G^{1})$. According to Propositions \ref{r45} and \ref{r51}, $G^{j+1}$ is a grid of balls $G^{j+1}= (u^{j+1},o^{j+1},n_{G} ,r^{j+1})$ with $D(\mathcal{G}_{i_{1}} \circ ... \circ \mathcal{G}_{i_{j}}  )_{o^{1}}   \in A^{j} \cdot (\mathcal{U}''_{x})^{p}$, there exist $(n+1)$ grids of balls $\Gamma^{j+1}_{p}$ (for $0 \le p \le n$) included in $(\mathcal{G}_{i_{1}} \circ ... \circ \mathcal{G}_{i_{j}})(\Gamma^{1}_{p})$ and $r^{j+1}$, $s^{j+1} $ satisfy the inequalities of property 3. The grids of balls $G^{j+1}$ and  $\Gamma^{j+1}_{p}$ (for $0 \le p \le n$) satisfy (1),(2),(3),(4).\newline \newline
Since $|a|\cdot R \cdot \max_{1 \le j \le q} (|| \mathcal{G}_{j} ||_{C^{2}}   ) < \frac{\nu \cdot r^{1}}{100}$ we have for every $0 \le p \le n$, for every $j \ge 1$ the following bounds : $r^{j},s^{j}_{p}>\frac{\nu \cdot r^{1}}{10}$. Since $D(\mathcal{G}_{i_{1}} \circ ... \circ \mathcal{G}_{i_{j}}  )_{o^{1}}     \in A^{j} \cdot (\mathcal{U}''_{x})^{p}$ and $p_{j} = p'$ is chosen according to Proposition \ref{r45} (see Proposition \ref{r45} for the definition of $p'$), we have for every $i$ such that $(\mathcal{G}_{i_{1}} \circ ... \circ \mathcal{G}_{i_{j}} \circ \mathcal{G}_{i})(G^{1}) \subset  \Gamma^{j}_{p_{j}}$ that $(\mathcal{G}_{i_{1}} \circ ... \circ \mathcal{G}_{i_{j}} \circ \mathcal{G}_{i})(G^{1})$ contains a grid of balls $(u^{j,i},o^{j,i},n_{G},r^{j,i})_{j}$ with $D(\mathcal{G}_{i_{1}} \circ ... \circ \mathcal{G}_{i_{j}} \circ \mathcal{G}_{i} )_{o^{1}}   \in A^{j+1} \cdot  \mathcal{U}_{x}$, this means that $(Q)$ is verified.
\newline \newline 
After Case 2, it follows from Proposition \ref{r45} that necessarily the two conditions of the following Case 3 are satisfied. \newline \newline
 $  \text{\underline{Case 3}: }   D(\mathcal{G}_{i_{1}} \circ ... \circ \mathcal{G}_{i_{j-1}} )_{o^{1}}    \in A^{j-1} \cdot (\mathcal{U}''_{x})^{p} \text{   and    }p_{j-1} \neq 0 $ \newline \newline Induction shows that $p_{j-1} $ had been chosen to get special composition properties (see Case 2, beware that the number denoted here by $p_{j-1}$ corresponds to the number denoted by "$p_{j}$" in Case 2), let us pick $p_{j} = 0$. By construction, $\mathfrak{C}$ intersects in its middle part a ball of the grid of balls $\Gamma^{j-1}_{p_{j-1}}$. Once again : $\overline{u^{j}} \in A^{j-1} \cdot \mathcal{N}(u^{1})$ and we have according to Corollary \ref{r55} that $\mathfrak{C}$ intersects in its middle part the ball of $\Gamma^{j}_{0}$ included inside $ (\mathcal{G}_{i_{1}} \circ ... \circ \mathcal{G}_{i_{j}})(\mathfrak{B})$ with $\mathcal{G}_{i_{j}}$ quasi-linear of type $p_{j-1}$. In particular, $\mathfrak{C}$ intersects the hull of a new grid of balls $G^{j+1} \subset (\mathcal{G}_{i_{1}} \circ ... \circ \mathcal{G}_{i_{j}})(G^{1})$. Once again, we can construct grids of balls $G^{j+1}$ and  $\Gamma^{j+1}_{p}$ (for $0 \le p \le n$) which satisfy (1),(2),(3),(4) but this time with $D(\mathcal{G}_{i_{1}} \circ ... \circ \mathcal{G}_{i_{j}}  )_{o^{1}}   \in A^{j} \cdot \mathcal{U}_{x}$. In particular, $\mathfrak{C}$ intersects in its middle part a ball of $\Gamma^{j}_{p_{j}}$. Moreover, Proposition \ref{r45} still insures that $(Q)$ is verified. \newline \newline 
Since $D(\mathcal{G}_{i_{1}} \circ ... \circ \mathcal{G}_{i_{j}}  )_{o^{1}}   \in  A^{j} \cdot \mathcal{U}_{x}$, we are now in Case 1 once again.  \qedhere

\end{proof}
\begin{proof}[Proof of Proposition \ref{r99999}]
According to the previous lemma, for each $j \ge 1$, $\mathfrak{C}$ intersects $\Gamma^{j}_{p_{j}} \subset (  \mathcal{G}_{i_{1}} \circ ... \circ \mathcal{G}_{i_{j-1}})(\Gamma^{1}_{p_{j}})$. But $\Gamma^{1}_{p_{j}} \subset \bigcup_{ 1 \le j \le q}  \mathcal{G}_{j}  (\mathfrak{B}_{p_{j}}) \subset \bigcup_{ 1 \le j \le q}  \mathcal{G}_{j}  (\mathfrak{B}) $. This shows that for each $j \ge 1$, there exist $q \ge i_{1},...,i_{j} \ge 1$ such that $\mathfrak{C}$ intersects $ (  \mathcal{G}_{i_{1}} \circ ... \circ \mathcal{G}_{i_{j}})(\mathfrak{B})$. This implies that $\mathfrak{C}$ intersects the limit set of the IFS $(\mathcal{G}_{1},...,\mathcal{G}_{q})$.  \qedhere

\end{proof}

\section{Properties of Latt\`es maps} \subsection{Definitions}
\begin{df3}A Latt\`es map is a holomorphic endomorphisms of $\mathbb{P}^{2}(\mathbb{C})$ of degree $d \ge 2$ which is semi-conjugate to an affine map on the torus. For such a map, we have the following commutative diagram : 

$$\xymatrix{
    \mathbb{T} \ar[r]^{\mathcal{L}} \ar[d]_\Pi & \mathbb{T} \ar[d]^\Pi \\
    \mathbb{P}^{2}(\mathbb{C}) \ar[r]_L & \mathbb{P}^{2}(\mathbb{C})
  }$$ where $\mathbb{T}$ is a complex torus of dimension $2$, $\Pi$ is a ramified covering of the projective space $\mathbb{P}^{2}(\mathbb{C})$ by the torus $\mathbb{T}$ and $\mathcal{L}$ is an affine map. \end{df3} \begin{pr3} The periodic points of any Latt\`es map are dense in $\mathbb{P}^{2}(\mathbb{C})$. The Julia set of any Latt\`es map is equal to $\mathbb{P}^{2}(\mathbb{C})$.  \end{pr3}
  
  \begin{nt5}
  In the following, for every $\tau \in \mathbb{C}$ such that $\mathrm{Im}(\tau)>0$, we will denote $L(\tau)$ the lattice in $\mathbb{C}$ given by : $L(\tau) = \mathbb{Z}+\tau \cdot \mathbb{Z}$ and by $L^{2}(\tau)$ the associated product lattice $L^{2}(\tau) = L(\tau) \cdot \begin{pmatrix}    1 \\
   0 \\ \end{pmatrix} + L(\tau) \cdot \begin{pmatrix}    0 \\
        1 \\ \end{pmatrix}$. We also let $\xi = e^{i\frac{2 \pi}{6}}$.
  
  \end{nt5}
  The following proposition can be found in \cite{jp}.
  \begin{pr3} \label{r31}
  If an affine map on a torus $\mathbb{T}$ induces a Latt\`es map $L$ on $\mathbb{P}^{2}(\mathbb{C})$, then the torus $\mathbb{T}$ is of the form $\mathbb{C}^{2}/\Lambda$ where $\Lambda$ is one of the six following lattices and the projection $\Pi : \mathbb{T} \rightarrow \mathbb{P}^{2}(\mathbb{C})$ is given (in some affine chart for Cases 1,2,3,4) by the following formulas : 
  \begin{itemize}
  \item[Case 1 ] $\Lambda = L^{2}(\tau) , (x,y) \mapsto [\wp(x)+\wp(y) : \wp(x)\wp(y) : 1]$
  \item[Case 2 ] $\Lambda = L^{2}(\xi), (x,y) \mapsto [\wp'(x)+\wp'(y) : \wp'(x)\wp'(y) : 1]$
  \item[Case 3 ] $\Lambda = L^{2}(i), (x,y) \mapsto [\wp^{2}(x)+\wp^{2}(y) : \wp^{2}(x)\wp^{2}(y) : 1]$
   \item[Case 4 ] $\Lambda = L^{2}(\xi), (x,y) \mapsto  [(\wp')^{2}(x)+(\wp')^{2}(y) : (\wp')^{2}(x)(\wp')^{2}(y) : 1]  $
   \item[Case 5 ]$ \Lambda = L^{2}(i), (x,y) \mapsto    [( \wp(x)\wp(y)+e_{1}^{2})^{2} : (\wp(x)+\wp(y))^{2} : (\wp(x)\wp(y)-e_{1}^{2})^{2} ] $
    \item[Case 6 ] $\Lambda = L(\tau) \cdot \begin{pmatrix}    -1 \\
   1 \\ \end{pmatrix} + L(\tau) \cdot \begin{pmatrix}    \xi^{2} \\
   \xi \\ \end{pmatrix}, (x,y) \mapsto  [\wp'(x_{1})-\wp'(y_{1}) : \wp(x_{1})-\wp(y_{1}) : \wp'(x_{1})\wp(y_{1})-\wp(x_{1})\wp'(y_{1})] $ \end{itemize}

    where $e_{1} = \wp(\frac{1}{2})$ and $(x_{1},y_{1})$ is the function of $(x,y)$ given by  : 

$$x_{1}\begin{pmatrix}
   -1 \\    1 \\ \end{pmatrix} +y_{1}\begin{pmatrix}
   \xi^{2} \\    \xi \\ \end{pmatrix} = \begin{pmatrix} x \\
   y \\ \end{pmatrix} $$

In the following, we will denote $\pi$ the projection from $\mathbb{C}^{2}$ to $\mathbb{T}^{2} = \mathbb{C}^{2}/\Lambda$.
  \end{pr3}
  \begin{df3}\label{defffp}
A product in the sense of Ueda is a holomorphic map on $\mathbb{P}^{2}(\mathbb{C})$ such that there exists a Latt\`es example $\tilde{L}$ on $ \mathbb{P}^{1}(\mathbb{C})$ such that we have : $$L \circ \eta = \eta \circ (\tilde{L},\tilde{L}) $$ where $\eta$ is the map between $\mathbb{P}^{1}(\mathbb{C}) \times \mathbb{P}^{1}(\mathbb{C})$ and $\mathbb{P}^{2}(\mathbb{C})$ which is just the projectivization of $(x,y) \mapsto (x+y,xy)$, given by : $$\eta : ([x : x'],[y:y']) \mapsto [xy'+x'y:xy:x'y']$$ Such a map $L$ is semi-conjugate to an affine map on the complex torus $\mathbb{T}$ and is a Latt\`es map.
\end{df3} 
Latt\`es maps corresponding to Cases 1,2,3 and 4 of Proposition \ref{r31} are products in the sense of Ueda. The following technical result was shown in \cite{sj} (Theorems 4.2 and 4.4). It will be used in the proof of Proposition 3.3.1.  
\begin{pr3} \label{uweb} For any Latt\`es map $L$ on $\mathbb{P}^{2}(\mathbb{C})$, one of the following is true : \begin{enumerate} \item either one map in $\{L,L^{2},L^{3}\}$ is a product in the sense of Ueda \item either one of the maps $L^{k}$ in $\{L,L^{2},L^{3},L^{6}\}$ is preserving an algebraic web associated to a smooth cubic (see \cite{dj} for this notion) \end{enumerate}  \end{pr3}

The following is an easy consequence of Propositions 3.1 to 3.6 of \cite{sj}.
  
  \begin{pr3} \label{r33}
   Let $\Lambda$, $\Pi$ be one of the lattices and associated coverings defined in Proposition \ref{r31}. There exists a finite group of unitary matrices $G_{\text{Latt\`es}}= G_{\text{Latt\`es}} (\Lambda,\Pi)$ of finite order such that every Latt\`es map has its linear part of the form $aA$ where $a \in \mathbb{C}^{*}$, $|a|\ge1$ and $A \in G_{\text{Latt\`es}}$.   
  \end{pr3}
  \begin{rk} Here, the scaling factor $a$ takes discrete values. Moreover, arbitrarily large values of $|a|$ can be obtained (it can be easily seen by taking the composition of a Latt\`es map with itself). The equality of the two topological degrees gives : $(d')^{2} = |a|^{4} \cdot  |\mathrm{det}(A)|^{2} $ where $d'$ is the algebraic degree of $L$.

  \end{rk}
  Since according to the previous result, there are only finitely many possible linear parts $A$ for a Latt\`es map (up to multiplication by the factor $a$) which are all of finite order, we can define the following integer.
  \begin{df5} \label{order} We denote by $\mathrm{ord}_{\mathrm{Latt\grave{e}s}}$ the product of all the orders of the possible linear parts $A$ for a Latt\`es map. \end{df5}

  It can be found in \cite{sj} that $\mathrm{ord}_{\text{Latt\`es}}$ is equal to $6^{2} \cdot 8^{2} \cdot 12 \cdot 24$. In a first reading, we encourage the reader to consider only the case where the linear part of the Latt\`es map is equal to $\mathrm{Id}$. In the other cases, the dynamical ideas are the same but with a few additional technicalities from algebra. In particular, it is sufficient in order to prove in some cases the corollary of the main result (see subsection 1.2).
  \subsection{An algebraic property of Latt\`es maps}
  
  The goal of this subsection is to prove the following result. 
  \begin{pr3} \label{r35}

For every torus $\mathbb{T}$, there exists an integer $i = i(\mathbb{T}^{2})$ such that for any $k>0$, there exists an integer $d_{ k}>0$ such that for any Latt\`es map $L$ of algebraic degree $d > d_{k}$, coming from an affine map on $\mathbb{T}$, there exists a homogenous change of coordinates $\varphi$ on $\mathbb{P}^{2}(\mathbb{C})$ such that :  $\varphi ^{-1} \circ L \circ \varphi$ is a holomorphic endomorphism of $\mathbb{P}^{2}(\mathbb{C})$ of the form $[\overline{P}_{1} : \overline{P}_{2} : \overline{P}_{3}]$ where the polynomial $\overline{P}_{3}$ is a product of irreducible factors $\overline{P}_{3,j}$ such that at least $k$ factors $\overline{P}_{3,j}$ are of degree bounded by $ i$.

\end{pr3}

\begin{df5} Let $v$ be a vector of $\mathbb{C}^{2}$ which belongs to a lattice $\Lambda$ and $v_{0} \in \mathbb{C}^{2}$. We suppose that the action of $\Lambda$ upon $\mathbb{C} \cdot v$ by translation is cocompact. Let $\mathbb{T}^{2} = \mathbb{C}^{2}/\Lambda$ and $\pi : \mathbb{C}^{2} \rightarrow \mathbb{T}^{2}$ be the natural projection. Then, then we say that $\pi(v_{0} + \mathbb{C} \cdot v)$ is a compact line of the torus $\mathbb{T}$ of direction $v$. It is compact and $\pi(v_{0} + \mathbb{C} \cdot v) \simeq \pi(v_{0}) + \mathbb{C}/\Lambda'  \cdot v$ for some subgroup $\Lambda' \subset \Lambda $. The family of compact lines of the torus $\mathbb{T}$ of direction $v$ is the family of all the compact lines of the torus of direction $v$ obtained by varying $v_{0}$.

\end{df5}

Let us point out the fact that $v \in \Lambda$ is not sufficient to conclude that the action of $\Lambda$ upon $\mathbb{C} \cdot v$ by translation is cocompact.

\begin{pr3} \label{com} Let $\Lambda$, $\Pi$ be one of the lattices and associated coverings defined in Proposition \ref{r31}.  Let $v$ be a vector of $\mathbb{C}^{2}$ which belongs to $\Lambda$ such that the action of $\Lambda$ upon $\mathbb{C} \cdot v$ by translation is cocompact. The family of images under $\Pi$ of compact lines of direction $v$ on the torus $\mathbb{T}$ is a family of algebraic curves of $\mathbb{P}^{2}(\mathbb{C})$ of degree bounded by $i = i(v,\mathbb{T}^{2})$.

\end{pr3} \begin{proof}Let $\mathcal{F}$ be the family of images of compact lines of direction $v$ on the torus $\mathbb{T}$ under $\Pi$. The family $\mathcal{F}$ is a holomorphic compact family of compact curves so that by the GAGA principle it is an algebraic family of curves and in particular their degree is bounded by some $i=i(v,\mathbb{T}^{2})$
\end{proof}

\begin{pr3} \label{r32} Let $\Lambda$, $\Pi$ be one of the lattices and associated coverings defined in Proposition \ref{r31}. Then, there exists a line $\delta$ in $\mathbb{P}^{2}(\mathbb{C})$ such that $\Pi^{-1}(\delta)$ contains at least one compact line $\mathcal{D}$ of $\mathbb{T}$.  \end{pr3}

\begin{proof}  
  In each case, the following compact lines are convenient for $\delta$ and we give the preimage compact lines $\mathcal{D}$. The first four cases cover the case of a product in the sense of Ueda. \newline \newline 
  $\text{Case 1 : }  \delta = \{Y=0\}$ Indeed, $ Y=0$ if and only if $\wp(x)\wp(y) = 0$. $\Pi^{-1}(\{Y=0\})$ is an union of compact lines of the torus of the form $\{x_{0}\} \times \mathbb{T}^{1}$ and $\mathbb{T}^{1} \times \{y_{0}\}$ where the $x_{0},y_{0}$ are in $\wp^{-1}(\{0\})$. \newline \newline
   $\text{Cases 2, 3 and 4: }  \delta = \{Y=0\}$. The proof is similar to Case 1 with respectively  $\wp'(x)\wp'(y) = 0$, $\wp^{2}(x)\wp^{2}(y) = 0$ and $(\wp'(x))^{2}(\wp'(y))^{2} = 0$.\newline \newline $\text{Case 5: }  \delta = \{X=Z\}$   Indeed, $ X=Z$ if and only if $(\frac{\wp(x)\wp(y)+e_{1}^{2}}{\wp(x)\wp(y)-e_{1}^{2}})^{2} = 1$, this means if and only if $4e_{1}^{2}\wp(x)\wp(y) = 0$. $\Pi^{-1}(\{X=Z\})$ is an union of compact lines of the torus of the form $\{x_{0}\} \times \mathbb{T}^{1}$ and $\mathbb{T}^{1} \times \{y_{0}\}$ where the $x_{0},y_{0}$ are in  $\wp^{-1}(\{0\})$. \newline \newline $\text{Case 6: }  \delta = \{ Z=0\}$ Indeed, $ Z=0$ if and only if $ \wp'(x_{1})\wp(y_{1})-\wp(x_{1})\wp'(y_{1}) = 0$. $\Pi^{-1}(\{Z=0\})$ contains the compact line of the torus $\{x_{1} = y_{1}\}$ (in the coordinates $x_{1},y_{1}$). \newline \newline In all the cases, the preimage of $\delta$ by $\Pi$ contains a compact line of the torus.  \qedhere
   
   \end{proof}
   \begin{pr3} \label{r34} If an affine map $\mathcal{L}$ of linear part $aA$ on a torus $\mathbb{T}$ induces a Latt\`es map $L$ on $\mathbb{P}^{2}(\mathbb{C})$ and $\mathcal{D}$ is the preimage under $\Pi$ of the compact line $\delta$ given by Proposition \ref{r32}, then the preimage of $\mathcal{D}$ under $\mathcal{L}$ is a finite union of compact lines of the torus. Moreover, the number of possible directions is finite. For each $k>0$, there exists $d_{k}>0$ such that for every Latt\`es map $L$ of algebraic degree greater than $d_{k}$ induced by an affine map $\mathcal{L}$ on $\mathbb{T}$, there exist at least $k$ distinct irreducible components of $L^{-1}(\delta)$ of degree bounded by $i$.   
   \end{pr3}
  
  \begin{proof} From Proposition \ref{r33}, we know that the linear part of $\mathcal{L}$ is of the form $aA$ with $A \in G_{\text{Latt\`es}}$. We denote by $\mathcal{L}_{\mathbb{C}^{2}}$ an affine map on $\mathbb{C}^{2}$ which induces the affine map $\mathcal{L}$ on $\mathbb{T}$.  The linear part of $\mathcal{L}_{\mathbb{C}^{2}}$ is $aA$. We know that $\mathcal{D}$ is a compact line of the torus $\mathbb{T}$ of direction $v$ (for some vector $v$ of $\mathbb{C}^{2}$) and the preimage of $\mathcal{D}$ under the natural projection $\pi :\mathbb{C}^{2} \rightarrow \mathbb{T}$ is an union of lines of $\mathbb{C}^{2}$ of direction $v$. Since $\mathcal{D}$ is a compact line of $\mathbb{T}$, by definition, the action of $\Lambda$ on $\mathbb{C} \cdot v$ is cocompact. This is equivalent to the existence of two complex numbers $\omega_{1}$ and $\omega_{2}$ which are not $\mathbb{R}$-colinear and such that $\omega_{1} v \in \Lambda$, $\omega_{2} v \in \Lambda $. We fix $\omega_{1}$ and $\omega_{2}$. We have $aA \cdot \Lambda \subset \Lambda$ because $aA$ is the linear part of $\mathcal{L}$. Then $a^{2}A^{2} \cdot \Lambda \subset \Lambda, \ldots, a^{\mathrm{ord}(A)-1}A^{\mathrm{ord}(A)-1} \cdot \Lambda \subset \Lambda$ (here $\mathrm{ord}(A)$ is the order of $A$, we know that $A$ is of finite order because it belongs to the finite group $G_{\text{Latt\`es}}$). But $a^{\mathrm{ord}(A)-1}A^{\mathrm{ord}(A)-1} = a^{\mathrm{ord}(A)}(aA)^{-1}$. Then $(aA)^{-1} ( a^{\mathrm{ord}(A)} \omega_{1} v ) \in \Lambda$ and $ (aA)^{-1} ( a^{\mathrm{ord}(A)}  \omega_{2} v ) \in \Lambda$. For every line $\Delta$ of $\mathbb{C}^{2}$ of direction $v$, the preimage of $\Delta$ under $\mathcal{L}_{\mathbb{C}^{2}}$ is a line of $\mathbb{C}^{2}$ of direction $(aA)^{-1}(v)$. The two complex numbers $a^{\mathrm{ord}(A)}  \omega_{1}$ and $a^{\mathrm{ord}(A)}  \omega_{2}$ are not $\mathbb{R}$-colinear and satisfy $a^{\mathrm{ord}(A)} \omega_{1}   \cdot  (aA)^{-1} ( v ) \in \Lambda$ and $  a^{\mathrm{ord}(A)}  \omega_{2} \cdot   (aA)^{-1} (v ) \in \Lambda$ and then the action of $\Lambda$ on $\mathbb{C} \cdot (aA)^{-1}(v)$ is cocompact. Since $\pi \circ \mathcal{L}_{\mathbb{C}^{2}} = \mathcal{L} \circ \pi $, the preimage of $\mathcal{D}$ under $\mathcal{L}$ is an union of compact lines of $\mathbb{T}$ which all have the same direction. $G_{\text{Latt\`es}}$ is finite (see Proposition \ref{r33}) and so the possible number of directions is finite. Let $\mathcal{D}'$ be a preimage of $\mathcal{D}$ under $\mathcal{L}$. \newline \newline We denote by $\Omega$ a fundamental domain of the action of $\Lambda$ on $\mathbb{C} \cdot (aA)^{-1}(v)$. We have the following straightforward property : for every lines $\Delta_{1}, \Delta_{2} $ of direction $(aA)^{-1}(v)$ in $\mathbb{C}^{2}$, we have $\pi(\Delta_{1}) = \pi(\Delta_{2})$ if and only if every two points respectively in $\Delta_{1}$ and $\Delta_{2}$ are joined by a vector which lies in $\Omega +\Lambda$. Let us take $\lambda_{1}\in \Lambda$ such that $(aA)^{-1}\lambda_{1}$ is not $\mathbb{C}$-colinear to $\mathcal{D}'$. Then, there is some constant $a_{k}>0$ such that if $|a|>a_{k}$, then the $100k$ vectors $ (aA)^{-1}(\lambda_{1}), 2 (aA)^{-1}(\lambda_{1}), \ldots, 100 (aA)^{-1}(\lambda_{1}) $ do not belong to $   \Omega +\Lambda$. Then the $100k$ images of $\mathcal{D}'$ by translations of vectors $(aA)^{-1}(\lambda_{1}), 2 (aA)^{-1}(\lambda_{1}), \ldots, 100 (aA)^{-1}(\lambda_{1})$ are 100 distinct preimages of $\mathcal{D}$ under $\mathcal{L}$ and they are compact lines of $\mathbb{T}$ of same direction $\mathbb{C} \cdot (aA)^{-1}(v)$.  \newline \newline Their images under $\Pi$ are irreducible components of degree bounded by $i=i(v,\mathbb{T})$ by Proposition \ref{com}. If $|a|>a_{k}$, at least $k$ (this term $k$ is not optimal and we get it by projection of the previous $100k$ lines) of them are distinct preimages of $\delta$ under $L$. But $|a|>a_{k}$ if $\text{deg}(L)$ is superior to some value $d_{k,A}$ (see the remark after Proposition \ref{r33}). Then, it suffices to take for $d_{k}$ the maximal value of $d_{k,A}$ when varying $A$ in $G_{\text{Latt\`es}}$ (see Proposition \ref{r33}).
    \end{proof}
  
  \begin{proof}[Proof of Proposition \ref{r35}]

 Let $\delta$ be a line in $\mathbb{P}^{2}(\mathbb{C})$ as in Proposition \ref{r32}. The result is a consequence of Proposition \ref{r34} because after a suitable change of coordinates, we can take $\delta = \{Z = 0\}$. Then $\{\overline{P}_{3} = 0\}$ contains at least the $k$ irreducible components of degree bounded by $i$ which are preimages of $\delta$ by $L$.  \qedhere

\end{proof}

\subsection{A periodic orbit in the postcritical set}

  Remind that the integer $\mathrm{ord}_{\text{Latt\`es}}$ was defined in Definition \ref{order}. Beware that in the following, the period of a periodic point is the exact period.
  
  \begin{pr3} \label{r41} There exists an integer $K>0$ such that for every Latt\`es map $L$  defined on $\mathbb{P}^{2}(\mathbb{C})$, there exists a point $c$ in the critical set of $L$ which is sent after $n_{c}$ iterations on a periodic point $p_{c}$ of period $n_{pc}$ such that : \begin{enumerate} \item $n_{c}+n_{pc} \le K$ \item $n_{p_{c}}$ is a multiple of $\mathrm{ord}_{\mathrm{Latt\grave{e}s}}$ \end{enumerate} 
\end{pr3}

\begin{proof} Let us start with the case of one dimensional Latt\`es maps.

\begin{lem3}Let $\tilde{L}$ be a one-dimensional Latt\`es map. There exists a critical point $\tilde{c}$ of $L$ which is sent after $\tilde{n}_{c} \le 12$ iterations on a periodic point $\tilde{p}_{c}$ of period $\tilde{n}_{pc} \le 12$. 
\end{lem3}

 \begin{proof}
The Lattes map $\tilde{L}$, according to Lemma 3.4 of \cite{lpp555}, is such that the postcritical set $P_{\tilde{L}}$ of $\tilde{L}$ is entirely included inside the set of critical values of the covering $\Theta$ of $\mathbb{P}^{1}(\mathbb{C})$ by the complex torus $\mathbb{T}^{1}$. This implies that every critical point of $\tilde{L}$ is sent after one iteration inside the set of the critical values of $\Theta$. Moreover, let us bound from above the number of critical values. This number $c_{r}$ is bounded from above by the number of critical points (counted with multiplicity). Still according to \cite{lpp555}, $\Theta$ can only be a covering of orders  $\text{ord}(\Theta) = 2,3,4 \text{ or } 6$. The Riemann-H\"urwitz formula gives us that : $\chi(\mathbb{T}^{1}) = \text{ord}(\Theta)\chi(\mathbb{P}^{1}(\mathbb{C})) - c_{r}$ which implies $c_{r} = 2 \cdot \text{ord}(\Theta)$. In particular, this means that the image of every critical point $\tilde{c}$ of $\tilde{L}$ is sent after $\tilde{n}_{c} \le 12$ iterations on a periodic point $\tilde{p}_{c}$ of period $\tilde{n}_{pc} \le 12$ . \end{proof}
\begin{lem3}Let $L$ be a product in the sense of Ueda. There exists a point $c$ in the critical set of  $L$ which is sent after $n_{c} \le 12$ iterations on a periodic point $p_{c}$ of period $n_{pc} \le 24 \cdot \mathrm{ord}_{\mathrm{Latt\grave{e}s}}$ which is a multiple of $\mathrm{ord}_{\mathrm{Latt\grave{e}s}}$. In particular, we have : $n_{c}+n_{pc} \le 12+24 \cdot \mathrm{ord}_{\mathrm{Latt\grave{e}s}}$. 
\end{lem3}
\begin{proof}

We take a critical point $\tilde{c}$ of $\tilde{L}$ given by the previous lemma. We take a periodic point $\tilde{p}$ of $\tilde{L}$ of period $2 \cdot \text{ord}_{\text{Latt\`es}}$ (it can be found in \cite{bker} that such a point actually exists because any rational map on  $\mathbb{P}^{1}(\mathbb{C})$  of degree greater than 2 has a point of strict period $2 \cdot \text{ord}_{\text{Latt\`es}}>4$ ). Then the point $c = \eta(\tilde{c},\tilde{p})$ is a critical point of $L$ (remind that $\eta$ was defined in Definition \ref{defffp}). It is sent after $n_{c} \le 12$ iterations on a periodic point $\eta(\tilde{p}_{c},\tilde{L}^{n_{c}}(\tilde{p}))$. The period of $\tilde{p}_{c}$ is $\tilde{n}_{pc} \le 12$ and $\tilde{p}$ is of period $2 \cdot \text{ord}_{\text{Latt\`es}}$. This implies that in $\mathbb{P}^{1}(\mathbb{C}) \times \mathbb{P}^{1}(\mathbb{C})$, the periodic point  $(\tilde{p}_{c},\tilde{L}^{n_{c}}(\tilde{p}))$ for $(\tilde{L},\tilde{L})$ is of period a multiple of $2 \cdot \text{ord}_{\text{Latt\`es}}$ bounded by $24 \cdot \text{ord}_{\text{Latt\`es}}$. Since the map $\eta$ is a two-covering, in $\mathbb{P}^{2}(\mathbb{C})$, the periodic point $\eta(\tilde{p}_{c},\tilde{L}^{n_{c}}(\tilde{p}))$ for $L$ is of period $n_{pc}$ which is a multiple of $\text{ord}_{\text{Latt\`es}}$ bounded by $24 \cdot \mathrm{ord}_{\text{Latt\`es}}$. Then $n_{c}+n_{pc} \le 12+24 \cdot \mathrm{ord}_{\text{Latt\`es}}$.
\end{proof}

Let us now prove Proposition \ref{r41}. According to Proposition \ref{uweb}, we have :  \begin{enumerate} \item Either one map in $\{L,L^{2},L^{3}\}$ is a product in the sense of Ueda. In this first case, the previous lemma shows that one of the maps in $\{L,L^{2},L^{3}\}$ has a point of its critical set which is sent after at most 12 iterations onto a periodic point of period a multiple of $\text{ord}_{\text{Latt\`es}}$ bounded by $24 \cdot \text{ord}_{\text{Latt\`es}}$. This implies that there exists a critical point of $L$ which is sent after $n_{c}$ iterations onto a periodic point of period $n_{pc}$ which is a multiple of $\text{ord}_{\text{Latt\`es}}$ with $n_{c}+n_{pc} \le 3 \cdot (12+24 \cdot \text{ord}_{\text{Latt\`es}})$.  \item One of the maps $L^{k}$ in $\{L,L^{2},L^{3},L^{6}\}$ is preserving an algebraic web associated to a smooth cubic. This implies (see the remark after Theorem A in \cite{dj}) that the critical set of $L^{k}$ is sent after one iteration into the set of critical values of $\Pi$ which is a curve $\mathcal{PC}$. In this second case, we have that $L^{k}(\mathcal{PC}) \subset \mathcal{PC}$ and $L^{k}$ induces by restriction a map on $\mathcal{PC}$. Taking the normalization of $\mathcal{PC}$ if necessary, we can suppose that $\mathcal{PC}$ is regular. There are two possibilities. Either $\mathcal{PC}$ is isomorphic to $\mathbb{P}^{1}(\mathbb{C})$ and $L^{k}$ induces a rational map so it has a periodic point of period $\text{ord}_{\text{Latt\`es}}$(again, it can be found in \cite{bker} that such a point actually exists). Either $\mathcal{PC}$ is isomorphic to a complex torus and $L^{k}$ induces a multiplication on this torus which also has a periodic point of period $\text{ord}_{\text{Latt\`es}}$. In both cases, we see that $L$ has a critical point which is sent after at most 6 iterations on a point of period a multiple of $\text{ord}_{\text{Latt\`es}}$ bounded by $6 \cdot \text{ord}_{\text{Latt\`es}}$.  
\end{enumerate}
Then, taking $K =  \max(3 \cdot (12+24 \cdot \text{ord}_{\text{Latt\`es}}),6+6 \cdot \text{ord}_{\text{Latt\`es}})$, the proof of the proposition is done.  \qedhere

\end{proof}

\section{Perturbations of Latt\`es maps}
\subsection{Some useful lemmas} In this subsection, we prove two lemmas about complex analysis. The constants which are involved in these lemmas will be fixed in the two next subsections. The following first lemma will be used in 4.2.8, 4.3.20 and in the proof of Lemma 4.4.8.
\begin{lem3} \label{r75}
For every $m>0$, for every ball $\tilde{\mathbb{B}}$, for every $1>\psi_{1}>0$, $1>\psi_{2}>0$, there exist constants $\rho=\rho(m,\tilde{\mathbb{B}})>0$, $\sigma=\sigma(m,\tilde{\mathbb{B}})>0$ such that for every rational function $h$ of degree equal to $m$, there exists a ball $\mathbb{B} \subset  \tilde{\mathbb{B}} \subset \mathbb{C}^{2}$ of radius larger than $\rho$ such that : \begin{equation}     \label{eq3} \forall z \in \mathbb{B},  \frac{|| Dh(z) ||}{|h(z)|} \le \sigma \end{equation}  \end{lem3} \begin{equation}     \label{eq31} \forall (z,z') \in \mathbb{B}^{2},  \frac{| h(z) |}{|h(z')|} \le 1+\psi_{1} \end{equation}  \begin{equation}     \label{eq33} \forall (z,z') \in \mathbb{B}^{2},  \text{arg}(h(z)) -\text{arg}( h(z')) \le \psi_{2} \end{equation} 
The lemma will be a consequence of the following lemma.
\begin{lem3} \label{rbie}
For every $m>0$, for every ball $\tilde{\mathbb{B}}$, there exist constants $\overline{\rho}=\overline{\rho}(m,\tilde{\mathbb{B}})>0$, $\tau=\tau(m,\tilde{\mathbb{B}})>0$ such that for every rational function $h$ of degree $m$, there exists a ball $\overline{\mathbb{B}} \subset  \tilde{\mathbb{B}} \subset \mathbb{C}^{2}$ of radius larger than $\overline{\rho}$ such that :  $$  \frac{  \mathrm{inf}_{\overline{\mathbb{B}}}  | h|}  {  \mathrm{sup}_{\overline{\mathbb{B}}}  |h|}\ge \tau$$
\end{lem3}
\begin{proof}
Let us denote $\mathcal{R}_{\text{norm}}$ the set of rational maps of degree $m$ which can be written $h = \frac{h_{1}}{h_{2}}$ where $h_{1}$ and $h_{2}$ are two polynomials whose coefficients $(a_{ij})$ and $(b_{ij})$ are such that : $\text{max}(a_{ij}) = \text{max}(b_{ij}) = 1$. $\mathcal{R}_{\text{norm}}$ is a compact set. For a given $h \in  \mathcal{R}_{\text{norm}}$, since $h \neq 0$, there exists  $\overline{\rho}_{h}>0$, $\tau_{h}>0$, a ball $\overline{\mathbb{B}}_{h} \subset    \tilde{\mathbb{B}}  \subset \mathbb{C}^{2}$ of radius $\overline{\rho}_{h}$ such that :  $$  \frac {  \text{inf}_{\overline{\mathbb{B}}_{h}}  | h|} {  \text{sup}_{\overline{\mathbb{B}}_{h}}  |h|} \ge \tau_{h}$$The constants $\overline{\rho}_{h}$ and $\tau_{h}$ can be chosen locally constant for rational functions in $\mathcal{R}_{\text{norm}}$ near $h$. Since $\mathcal{R}_{\text{norm}}$ is compact, if we choose $\overline{\rho}=\overline{\rho}(n, \tilde{\mathbb{B}})$ the minimum of the $\overline{\rho}_{h}$ and $\tau=\tau(n, \tilde{\mathbb{B}})$ the minimum of the $\tau_{h}$ for a finite covering of $\mathcal{R}_{\text{norm}}$, we have : for every rational map $h \in \mathcal{R}_{\text{norm}}$ of degree $m$, there exists a ball $\overline{\mathbb{B}} \subset \tilde{\mathbb{B}} \subset \mathbb{C}^{2}$ of radius larger than $\overline{\rho}$ such that :  $$  \frac{  \text{inf}_{\overline{\mathbb{B}}}  | h|} {  \text{sup}_{\overline{\mathbb{B}}}  |h|} \ge \tau$$ Since every rational map $h$ of degree $m$ can be written $h = C^{ste} \cdot \tilde{h}$ with $\tilde{h} \in \mathcal{R}_{\text{norm}}$, the result is true for every rational map of degree $m$.  \qedhere

\end{proof} \begin{proof}[Proof of Lemma \ref{r75}]
 We fix such a ball $\overline{\mathbb{B}}$. Up to multiplying $h$ by a constant, which does not affect (\ref{eq3}), we can suppose that $|h|_{\infty}=1$. We denote $ \frac{1}{\tau} = \sigma$ where $\tau$ comes from Lemma \ref{rbie}. Then by the Cauchy inequality we have : $\frac{|| Dh(z) ||}{|h(z)|} \le \frac{1}{\tau} = \sigma$, this is (\ref{eq3}). Then (\ref{eq31}) and (\ref{eq33}) are simple consequences of (\ref{eq3}). The lemma is proven.  \qedhere

\end{proof}

The following interpolation result will be used in 4.2.7, 4.2.10, 4.3.21 (and thus in the proof of Lemma 4.4.6) and in the proof of Lemma 4.4.7. Remind that $n$ was defined in Proposition \ref{r45}.
\begin{lem3} \label{r777} Let us take $n$ balls $\mathcal{V}_{1},...,\mathcal{V}_{n} \subset \mathrm{Mat}_{2}(\mathbb{C})$. There exists an integer $\tilde{d} = \tilde{d}(\mathcal{V}_{1},...,\mathcal{V}_{n})$ and two real numbers $1>\psi_{1}>0$ and $1>\psi_{2}>0$ such that for every $\xi>0$, there exists a constant $\nu=\nu(n,\xi)>0$ such that : for every ball $\mathbb{B} \subset \mathbb{C}^{2}$ of radius bounded by 1, for every $\theta_{0} \in \mathbb{R}$, there exist a polynomial map $H=H(\mathcal{V}_{1},...,\mathcal{V}_{n},\mathbb{B},\theta_{0})$ of $\mathbb{C}^{2}$ of degree $\tilde{d}$ and $(n+1)$ balls $\mathbb{B}_{0},...,\mathbb{B}_{n} \subset \mathbb{B}$ of radius greater than $\nu \cdot \text{rad}(\mathbb{B})$ such that on each $\mathbb{B}_{j} $ :    $$\forall t \in (1-\psi_{1},1+\psi_{1}), \forall \theta \in (\theta_{0}-\psi_{2},\theta_{0}+\psi_{2}) : - e^{i \theta} \cdot t \cdot DH \in \mathcal{V}_{j} \text{  and  }2 \cdot |H|_{\infty}< \xi$$

\end{lem3} \begin{proof} We call $\tilde{v}_{1},...,\tilde{v}_{n}$ the centers of the balls $\mathcal{V}_{1},...,\mathcal{V}_{n} \subset \mathrm{Mat}_{2}(\mathbb{C})$. Let us take the ball $\mathbb{B} = \mathbb{B}(0,1)$. For a given $\theta_{0} \in [0,2\pi]$, there exists $H$ having its differentials at $n$ points $p_{i} \in \mathbb{B}(0,1)$ satisfying $H(p_{i}) = 0$ and $DH_{p_{i}} = e^{-i.\theta_{0}}\cdot \tilde{v}_{j}$ by interpolation. Taking sufficiently small balls $\mathbb{B}_{1},...,\mathbb{B}_{n}$ of radius $\nu$ around the points $p_{i}$, this gives the result for a given $\theta_{0} \in [0,2\pi]$ and $t =1$. Moreover, since the required condition are open, $H$ can be taken uniform on a small interval of values of $\theta$ and a small interval $(1-\psi_{1},1+\psi_{1})$ of values of $t$. Then $\tilde{d}$, $\nu$ and $\psi_{1}$ can be taken locally constant in $\theta$. Since $[0,2\pi]$ is compact, we take the maximal value of $\tilde{d}$ and the minimal values of $\nu$ and $\psi_{1}$ on a finite covering of $[0,2\pi]$ by intervals where $\tilde{d}$, $\nu$ and $\psi_{1}$ can be taken constant on each interval of the covering. In particular, since this covering is finite, there exists $\psi_{2}>0$ such that for each $\theta_{0} \in [0,2\pi]$, it is possible to find constant $H,m,\nu,\psi_{1}$ for every $\theta \in (\theta_{0}-\psi_{2},\theta_{0}+\psi_{2})$. This gives us the result for the fixed ball $\mathbb{B}(0,1)$. Then, the result follows for any ball $\mathbb{B}(\gamma,r)$ with $r \le 1$ by taking the map $\tilde{H} = (r \cdot \text{Id} + \gamma ) \circ  H \circ (\frac{1}{r} \cdot  \text{Id}-\gamma)$. It is easy to check that :  $- e^{i \theta_{0}} \cdot t \cdot D\tilde{H} = - e^{i \theta_{0}} \cdot t \cdot DH \in \mathcal{V}_{j}$ and $2 \cdot |\tilde{H}|_{\infty}<2 \cdot  r \cdot |H|_{\infty}< \xi$.
 \end{proof} 

\subsection{Fixing the constants relative to the torus $\mathbb{T}$ and the matrix of the linear part $A$} In the two next subsections, we fix some notation and define a certain number of constants and objects in the following specified order. As a guide for the reader objets denoted in roman letters are relative to $\mathbb{P}^{2}(\mathbb{C})$, and gothic letters are relative to the torus.

\begin{enumerate} [label = \arabic*.]
\item We fix a torus $\mathbb{T}$ and euclidean coordinates $\pi : \mathbb{C}^{2} \rightarrow \mathbb{T}$. We fix the projection $\Pi : \mathbb{T} \rightarrow \mathbb{P}^{2}(\mathbb{C})$ as in Proposition \ref{r31}. We fix the group $G_{\text{Latt\`es}}=G_{\text{Latt\`es}}(\mathbb{T},\Pi)$ given by Proposition \ref{r33}.
\item We fix a Fubini-Study metric $||.||_{FS}$ on $\mathbb{P}^{2}(\mathbb{C})$.
\item We fix the matrix of the linear part $A$ with $A \in G_{\text{Latt\`es}}$. We fix a line $\delta$ as in Proposition \ref{r32}. We fix affine coordinates $[z_{1},z_{2},z_{3}]$ on $\mathbb{P}^{2}(\mathbb{C})$ as in Proposition \ref{r35} in which $\delta = \{z_{3} = 0\}$. In the following, we dehomogenize by working in the chart $ \{[z_{1},z_{2},z_{3}] : z_{3} \neq 0\}$ on $\mathbb{P}^{2}(\mathbb{C})$.
\item  We first need a proposition. \begin{nt5} We will denote : $V^{j} = v_{j} + r_{j} \cdot B(0,1)$ (remind the balls $V^{j}$ were defined in Proposition \ref{r45}). $\frac{1}{4} \cdot V^{j}$ will denote the ball of same center as $V^{j}$ and with quarter of radius. \end{nt5}
We fix $\mathfrak{p}_{0} \in \mathbb{T} $ such that $\Pi(\mathfrak{p}_{0}) \in \{[z_{1},z_{2},z_{3}] : z_{3} \neq 0\}$ and $D\Pi_{\mathfrak{p}_{0}}$ is invertible. There exist invertible matrices $M_{1},...,M_{n}$ such that for every $j$ :  
$$(D\Pi^{-1})_{\Pi(\mathfrak{p}_{0})} \cdot   \Big( D\Pi_{\mathfrak{p}_{0}} \cdot A \cdot (D\Pi^{-1})_{\Pi( \mathfrak{p}_{0})} \Big)^{-1}    \cdot M_{j} \cdot   \Big( D\Pi_{\mathfrak{p}_{0}} \cdot A \cdot (D\Pi^{-1})_{\Pi( \mathfrak{p}_{0})} \Big)^{-1}     \cdot D\Pi_{\mathfrak{p}_{0}} =v_{j}     $$  Then, by continuity we have : 
\begin{lem3} \label{r73}
There exists a ball $\tilde{\mathfrak{B}} = \tilde{\mathfrak{B}}(\mathbb{T},A) \subset \mathbb{T}$ (remind we have fixed euclidean coordinates on $\mathbb{T}$) where $\Pi$ is invertible such that $\Pi(\tilde{\mathfrak{B}}) \Subset \{[z_{1},z_{2},z_{3}] : z_{3} \neq 0\}$, a constant $\sigma'=\sigma'(\mathbb{T},A)>0$ and $n$ balls $\mathcal{V}_{1},...,\mathcal{V}_{n} \subset \mathrm{Mat}_{2}(\mathbb{C})$ with :  \scriptsize    $$  \forall \mathfrak{p}_{i} \in \tilde{\mathfrak{B}},   (D\Pi^{-1})_{\Pi(\mathfrak{p}_{1})} \cdot     \Big( D\Pi_{\mathfrak{p}_{2}} \cdot A \cdot (D\Pi^{-1})_{\Pi( \mathfrak{p}_{3})} \Big)^{-1}      \cdot \mathcal{V}_{j} \cdot    \Big( D\Pi_{\mathfrak{p}_{4}} \cdot A \cdot (D\Pi^{-1})_{\Pi( \mathfrak{p}_{5})} \Big)^{-1}     \cdot  D\Pi_{\mathfrak{p}_{6}}  \in \frac{1}{4} \cdot V^{j}    $$ \end{lem3}
\begin{lem3} \label{r74} Reducing $\tilde{\mathfrak{B}}$ if necessary, there exists a constant $\sigma'=\sigma'(\mathbb{T},A)>0$ such that for every $w$ with $||w|| = 1$, we have :  $$ \inf_{ \mathfrak{p} \in \tilde{\mathfrak{B}} }  ||D(\Pi \circ A \circ \Pi^{-1})_{\Pi(\mathfrak{p})}(w)||  \ge \sigma' \cdot  (\sup_{ \mathfrak{p} \in \tilde{\mathfrak{B}} }  ||D\Pi_{\mathfrak{p}} ||) \cdot ||A|| \cdot  ( \sup_{ \mathfrak{p} \in \tilde{\mathfrak{B}} }  ||(D\Pi^{-1})_{\Pi(\mathfrak{p})}|| )     $$

\end{lem3}
\begin{proof}  We take :  $$\sigma' = \frac{1}{2} \cdot  \frac{ \text{inf}(\text{Sp}((D(\Pi \circ A \circ \Pi^{-1})_{\Pi(\mathfrak{p}_{0})} ) }{   (\sup_{ \mathfrak{p} \in \tilde{\mathfrak{B}} }  ||D\Pi_{\mathfrak{p}} ||) \cdot ||A|| \cdot  ( \text{sup}_{ \mathfrak{p} \in \tilde{\mathfrak{B}} }  ||(D\Pi^{-1})_{\Pi(\mathfrak{p})}|| )   }  $$ and the condition holds reducing the size of the ball $\tilde{\mathfrak{B}}$ around $\mathfrak{p}_{0}$ if necessary.
\end{proof} 
We fix such a ball $\tilde{\mathfrak{B}}$, a constant $\sigma'>0$ and $n$ balls $\mathcal{V}_{1},...,\mathcal{V}_{n} \subset \mathrm{Mat}_{2}(\mathbb{C})$ of centers $\tilde{v}_{1},...,\tilde{v}_{n}$. 

\item We will use the following notation :  \begin{nt5} In the following, we still denote $||M||$ the norm $||.||_{2,2}$ of a fixed matrix. We will denote :  $$||D\Pi|| = \sup_{   \mathfrak{p} \in \tilde{\mathfrak{B}}} ||D\Pi_{\mathfrak{p}}||  \text{  and  } ||D\Pi^{-1}|| = \sup_{   \mathfrak{p} \in \tilde{\mathfrak{B}}} ||(D\Pi^{-1})_{\Pi(\mathfrak{p})}||$$  \end{nt5}

\item We fix $\tilde{\mathbb{B}}=\tilde{\mathbb{B}}(\mathbb{T},A,\tilde{\mathfrak{B}})$ a ball included in $\Pi(\tilde{\mathfrak{B}})$. There exists some constant $\iota>0$ such that for every ball $\mathbb{B} \subset \tilde{\mathbb{B}}$ of radius $r$, $\Pi^{-1}(\mathbb{B}) \cap \tilde{\mathfrak{B}}$ contains a ball of radius $\iota \cdot r$. We fix such a constant $\iota$. We take the restriction of $||.||_{FS}$ on $\tilde{\mathbb{B}}$. Since $\tilde{\mathbb{B}} \Subset \{[z_{1},z_{2},z_{3}] : z_{3} \neq 0\} $, this restriction is equivalent to the euclidean metric on $\{[z_{1},z_{2},z_{3}] : z_{3} \neq 0\}$.
\item We fix the integer $m=\tilde{d}$ and the reals $\psi_{1},\psi_{2}>0$ given by Lemma \ref{r777} associated to the balls $\mathcal{V}_{1},...,\mathcal{V}_{n} \subset \mathrm{Mat}_{2}(\mathbb{C})$.
\item We fix the constants : $$\rho=\rho(m)>0 \text{ and } \sigma=\sigma(m)>0$$ given by Lemma \ref{r75} associated to the integer $m=\tilde{d}$, the ball $\tilde{\mathbb{B}}$ and the two reals $\psi_{1},\psi_{2}$.
\item  We take a constant $\xi=\xi(\mathbb{T},\Pi,\mathcal{V}_{i},A,\sigma,\sigma')$ satisfying the following inequality : $$0<\xi < \frac{1}{4} \cdot \text{min}_{1 \le j \le n} r_{j} \cdot\text{min}(\frac{\sigma' \cdot ||A||^{2}}{2  ||\Pi^{-1}||_{C^{2}} \cdot ||D\Pi||},\frac{||A||^{2}}{2\sigma \cdot ||D\Pi||^{3} \cdot ||D\Pi^{-1}||^{3}})$$
\item From 4.2.7 and 4.2.8, Lemma \ref{r777} gives us a new constant $\nu=\nu(m,\xi)>0$. 
\item Corollary \ref{r55} gives us a constant $N(\frac{\nu \rho}{10})$.
\item We fix a constant $d^{1}$ defined as follows. Let us point out that for any Latt\`es map $L$ of algebraic degree $d'$ coming from an affine map $\mathcal{L}$ on $\mathbb{T}$, of linear part $aA$, the equality of the two topological degrees gives : $(d')^{2} =  |a |^{4} \cdot  |\text{det}(A) |^{2}$. There are $(d')^{2} =  |a |^{4} \cdot  |\text{det}(A) |^{2}$ disjoint preimages of the torus $\mathbb{T}$ by the affine map $\mathcal{L}$ of volume $\frac{\text{vol}(\mathbb{T})}{ |a|^{4} \cdot  |\text{det}(A) |^{2}}$. Let us denote $\text{vol}_{r}$ the volume of a ball of radius $r$. Let us take $d^{1}$ such that both $(d^{1})^{2} \cdot \frac{ \text{vol}_{\iota \cdot \rho}}{ 10 \cdot \text{vol}(\mathbb{T}) }>(\frac{10}{\nu} \cdot N(\frac{\nu \rho}{10}))^{4}$ and $(d^{1})^{2}  \ge 100 \cdot \text{max}_{A \in G_{\text{Latt\`es}}} |\text{det}(A) |^{2}$ (remind that $\iota$ was defined in 4.2.6). In particular, this last condition implies that for any Latt\` es map of algebraic degree $d' \ge d^{1}$, we have $|a| \ge 2$.
\item $i$ was defined in Proposition \ref{r35} and $K$ in Proposition \ref{r41}, we fix $n_{H} = E(\frac{m+2K}{i})+1$. We fix $d^{2} =  d_{n_{H}+100}$ (this integer was also defined in Proposition \ref{r35}). We fix $d^{3} = 2K$.
\item We fix $d = \text{max}(d^{1},d^{2},d^{3})$.

\end{enumerate}
\subsection{Fixing the constants relative to the Latt\`es map}

\begin{enumerate} [label = \arabic*., resume]
\item Let $L$ be a Latt\`es map $L = [\overline{P}_{1} : \overline{P}_{2} : \overline{P}_{3}]$ of degree $d'>d$ associated to an affine map on $\mathbb{T}$ of linear part $aA$. \item  According to Proposition \ref{r35}, in the coordinates $[z_{1},z_{2},z_{3}]$ which were fixed in 4.2.3. we have that : 
$$\overline{P}_{3}(z_{1},z_{2},z_{3}) =  \prod_{1 \le j \le J} \overline{P}_{3,j}(z_{1},z_{2},z_{3}) $$ with $P_{3,j}$ irreducible and $\text{deg}( \overline{P}_{3,i}   ) \le i = i(\mathbb{T}^{2})$ for $J \ge j \ge J - m-2K+1$ (remind that $K$ was defined in Proposition \ref{r41}). In plain words, the last factors of the product have degree bounded by a constant $i$ depending only on the chosen torus $\mathbb{T}$. We will consider the restriction of $L$ to $\{[z_{1},z_{2},z_{3}] : z_{3} \neq 0\} \cap L^{-1}(\{[z_{1},z_{2},z_{3}] : z_{3} \neq 0\})$. We have : $$L(z_{1},z_{2}) = (\frac{\overline{P_{1}}(z_{1},z_{2},1)} {\overline{P_{3}}(z_{1},z_{2},1)}, \frac{\overline{P}_{2}(z_{1},z_{2},1)}{\overline{P_{3}}(z_{1},z_{2},1)})$$ We denote $P_{i}(z_{1},z_{2}) = \overline{P}_{i}(z_{1},z_{2},1)$. 
\item There exists a periodic point $p_{c}$ of period $n_{pc}$ (which is a multiple of $\text{ord}_{\text{Latt\`es}}$) which belongs to the postcritical set of $L$, according to Proposition \ref{r41}, we fix it once for all. We call $c$ the point of the critical set such that $p_{c}$ is in the orbit of $c$ and we have $n_{c}+n_{pc} \le K$ according to Proposition \ref{r41}, where $K$ is independent of the choice of $\mathbb{T}$ and $L$. Since $p_{c}$ is repelling, we can suppose that $c$ is the only critical point in $\{c,L(c),...,p_{c},...,L^{n_{pc}-1}(p_{c})\}$. We choose homogenous polynomials of degree 1 denoted by $Q_{1},...,Q_{n_{c}+n_{pc}-1}$ such that :  \begin{equation} \label{eq73} Q_{1}(c) = Q_{2}(L(c))= ... = Q_{n_{c}}(p_{c}) = Q_{n_{c}+n_{pc}-1}(L^{n_{pc}-1}(p_{c})) = 0 \end{equation} It is possible to take these polynomials such that at least one of the coefficients of $z_{1}$ and $z_{2}$ is non equal to 0 so we take the polynomials with this property.
\item 

Putting $$ \tilde{P}_{3}(z_{1},z_{2}) =  \prod_{J-2(n_{c}+n_{pc})-m+1 \le j \le J } P_{3,j}(z_{1},z_{2}), $$
let us denote by $h$ the rational function defined by : 
$$h(z_{1},z_{2}) = \frac{\prod _{1 \le j \le n_{c}+n_{pc}-1}  (Q_{j}(z_{1},z_{2},1))^{2} }{ \tilde{P}_{3}(z_{1},z_{2}) }$$

\item We denote : $h(\Pi(\mathfrak{p}_{0})) = |h(\Pi(\mathfrak{p}_{0}))|e^{i \theta_{1}}$. 
\item We choose the ball $\mathbb{B} \subset \tilde{\mathbb{B}}$ of radius larger than $\rho$ according to Lemma \ref{r75} applied to the ball $\tilde{\mathbb{B}}$ chosen in 4.2.6 and to the the constants $m,\psi_{1},\psi_{2}$ chosen in 4.2.7. We pick a ball $\mathfrak{B} \subset \Pi^{-1}(\mathbb{B}) \cap \tilde{\mathfrak{B}}$ (remind that $\tilde{\mathfrak{B}}$ was defined in Lemma \ref{r73}). According to 4.2.6, $\mathfrak{B}$ can be taken with its radius equal to $\iota \cdot \rho$ and this bound on its radius (not the ball itself, but the bound on its radius) is independent of $L$. Since $d' \ge d \ge d^{1} $ with $(d^{1})^{2} \cdot \frac{ \text{vol}_{\iota \cdot \rho}}{ 10 \cdot \text{vol}(\mathbb{T}) }>(\frac{10}{\nu} \cdot N(\frac{\nu \rho}{10}))^{4}$, there are at least $(\frac{10}{\nu} \cdot N(\frac{\nu \rho}{10}))^{4}$ preimages of $\mathfrak{B}$ by the affine map $\mathcal{L}$ inside $\mathfrak{B}$ which form a grid of balls.
\item We fix the polynomial map $H=H(\mathcal{V}_{1},...,\mathcal{V}_{n},\mathbb{B},\theta_{0})$ of $\mathbb{C}^{2}$ of degree $m=\tilde{d}$ and $(n+1)$ balls $\mathbb{B}_{0},...,\mathbb{B}_{n} \subset \mathbb{B}$ given by Lemma \ref{r777} and corresponding to this ball $\mathbb{B}$ and the value $\theta_{0} = \theta_{1}-2\arg(a)$ where $\theta_{1}$ was defined in 4.3.19. Each of them has its radius larger than $\nu$ times the radius of $\mathbb{B}$. We take $(n+1)$ balls $\mathfrak{B}_{0} \subset \mathfrak{B} \cap \Pi^{-1}(\mathbb{B}_{0}),...,\mathfrak{B}_{n} \subset \mathfrak{B} \cap \Pi^{-1}(\mathbb{B}_{n})$ of radius $\iota \cdot ( \nu \cdot \text{rad}( \mathbb{B} ))$. Then the quotient $\frac{\text{rad}(\mathfrak{B}_{j})}{\text{rad}(\mathfrak{B})}$ is equal for each $j \in \{1,...,n\}$ to $\frac{\iota \cdot ( \nu \cdot \text{rad}( \mathbb{B} ))}{\text{rad}(\mathfrak{B})} = \frac{\iota \cdot ( \nu \cdot \rho  )}{ \iota \cdot \rho} = \nu$. Let us point out that this bound on the radius is still independent of $L$.

\end{enumerate}
\subsection{Creating a correcting IFS}

\begin{nt5}
 In the following we construct three holomorphic families of holomorphic maps of $\mathbb{P}^{2}(\mathbb{C})$ which are successive perturbations of $L$ : $L'=L'_{\varepsilon_{1}}$, $L'' = L''_{\varepsilon_{1},\varepsilon_{2}}$ and $L'''=L'''_{\varepsilon_{1},\varepsilon_{2},\varepsilon_{3}}$ where $\varepsilon_{1},\varepsilon_{2},\varepsilon_{3} \in \mathbb{D}$. We have $L'_{0} =L$, $L''_{\varepsilon_{1},0} = L'_{\varepsilon_{1}}$ and $L'''_{\varepsilon_{1},\varepsilon_{2},0} = L''_{\varepsilon_{1},\varepsilon_{2}}$. We often forget the $\varepsilon_{i}$ and just denote $L',L'',L'''$ for simplicity when there is no risk of confusion. \end{nt5} \begin{nt5}
 
 We consider the $q = q(d)$ preimages of $\Pi(\mathfrak{B})$ under $L$ included inside $\Pi(\mathfrak{B})$ and the corresponding local inverses $(g_{j})_{1 \le j\le q}$ of $L$. We denote by $(\mathcal{G}_{j} )_{1 \le j\le q}$ the corresponding maps on $\mathfrak{B}$. For further perturbations $L'$, $L''$, $L'''$ of $L$, we consider the analogous objects and we call them $(g'_{j})_{1 \le j\le q}$, $(g''_{j})_{1 \le j\le q}$, $(g'''_{j})_{1 \le j\le q}$ and $(\mathcal{G}'_{j} )_{1 \le j\le q}$, $(\mathcal{G}''_{j} )_{1 \le j\le q}$, $(\mathcal{G}'''_{j} )_{1 \le j\le q}$.

\end{nt5}

In the following, we will see that $(\mathcal{G}'''_{j} )_{1 \le j\le q}$ is a correcting IFS.

\begin{nt5} \label{notf}
In the following, we will consider the continuation $p(L')$ (resp. $p(L''),p(L''')$) of the periodic point $p_{c}$. This one is well defined according to the implicit function Theorem since $p_{c}$ is repelling. In fact, for the successive perturbations that we will consider, we will always have $p(L') =p(L'')=p(L''')=p_{c}$. \end{nt5}

 \begin{pr} \label{r71777}

Let $L$ be a Latt\`es map of degree $d'>d$ coming from an affine map on $\mathbb{T}$, of linear part $aA$. Let $L =  (\frac{P_{1}}{P_{3}} , \frac{ P_{2}}{P_{3}})$ be the expression of $L$ in the chart $\{[z_{1},z_{2},z_{3}] : z_{3} \neq 0\}$ defined in 4.2.3. Then the family of rational maps $(L'_{\varepsilon_{1}})_{\varepsilon_{1}}$ where $L' = L'_{\varepsilon_{1} }= (\frac{P'_{1}}{P_{3}} , \frac{ P'_{2}}{P_{3}})$ is defined by :  \begin{equation} \label{eq77777} P'_{1}(z_{1},z_{2}) =  P_{1}(z_{1},z_{2})+\varepsilon_{1} h(z_{1},z_{2})P_{3}(z_{1},z_{2})H_{1}(z_{1},z_{2}) \end{equation} \begin{equation} \label{eq55555}  P'_{2}(z_{1},z_{2}) = P_{2}(z_{1},z_{2})+\varepsilon_{1} h(z_{1},z_{2})P_{3}(z_{1},z_{2})H_{2}(z_{1},z_{2})  \end{equation} where $h$ was defined in 4.3.18, $H$ in 4.3.21 and $\varepsilon_{1} \in \mathbb{D}$ is such that : \begin{enumerate}

\item For every $\varepsilon_{1} \in \mathbb{D}$, $L' = L'_{\varepsilon_{1}}$ extends to a holomorphic map of $\mathbb{P}^{2}(\mathbb{C})$ of the same degree as $L$ and $(L'_{\varepsilon_{1}})_{\varepsilon_{1}}$ is a holomorphic family of holomorphic maps of $\mathbb{P}^{2}( \mathbb{C})$     \item  $p(L') =p_{c}$ is periodic for $L'$ and is in the forward orbit of $c$ : $p_{c} = (L'_{\varepsilon_{1}})^{n_{c}}(c)$ and $(L'_{\varepsilon_{1}})^{n_{pc}}(p_{c}) = p_{c}$. Moreover $D(L'_{\varepsilon_{1}})_{c} = D(L)_{c}, \cdots,  D(L'_{\varepsilon_{1}})_{(L'_{\epsilon_{1}})^{n_{pc}-1}(p_{c})} = DL_{L^{n_{pc}-1}(p_{c})} $ for every $\varepsilon_{1} \in \mathbb{D}$

 \end{enumerate}

\end{pr}
\begin{proof}

Let first remark that since $P_{3}$ admits at least $n_{H} =(E( \frac{m+2K)}{i})+1)$ factors of degree bounded by $i$, the degrees of $hP_{3}H_{1}$ and $hP_{3}H_{2}$ are bounded by $\text{deg}(P_{1})=\text{deg}(P_{2})$. Since the property of being a holomorphic mapping is open, $L'$ is a holomorphic mapping for sufficiently small values of $\varepsilon_{1}$. For simplicity we will suppose that this is true for $\varepsilon_{1} \in \mathbb{D}$ after rescaling if necessary. Since $\varepsilon_{1}$ is just a linear factor, $(L'_{\varepsilon_{1}})_{\varepsilon_{1}}$ is a holomorphic family of holomorphic maps of $\mathbb{P}^{2}( \mathbb{C})$. Thus item 1 is proven. Item 2 is a consequence of the quadratic terms $Q_{j}^{2}$ in $h$ (see 4.3.18).  \qedhere

\end{proof}

\begin{pr} \label{r71}

Let $L$ be a Latt\`es map of degree $d'>d$ coming from an affine map on $\mathbb{T}$, of linear part $aA$. We are working in the chart $\{[z_{1},z_{2},z_{3}] : z_{3} \neq 0\}$ defined in 4.2.3. In this chart, $L = (\frac{P_{1}}{P_{3}} , \frac{ P_{2}}{P_{3}})$. Let $L'$ as in Proposition \ref{r71777}. Then there exists $t>0$ such that for every $0 \le p \le n$, for every real $0< \varepsilon_{1} <1$, there exists a ball $\mathfrak{B}_{p} \subset \mathfrak{B} \subset \mathbb{C}^{2}$ of radius $\text{rad}(   \mathfrak{B}_{p})  \ge \nu \cdot \text{rad}(   \mathfrak{B})  $ and a neighborhood $\mathcal{X}_{\varepsilon_{1}}$ of $L'$ in $\mathrm{Hol}_{d'}$ such that for every $L''' \in \mathcal{X}_{\varepsilon_{1}}$, if $j$ is such that $\mathcal{G}'''_{j}(\mathfrak{B}) \subset \mathfrak{B}_{p}   $ then $\mathcal{G}_{j}'''  $ is quasi-linear of type $(t\varepsilon_{1},p)$ (remind that the notion of type was defined in Definition \ref{type}). 
 
\end{pr} \begin{proof}

 In the following, we omit the index $j$ on $g_{j}$, $g'_{j}$, $\mathcal{G}_{j}$ and $\mathcal{G}'_{j}$ and we take $0< \varepsilon_{1} <1$. Let us remind we work in the chart : $[z_{1},z_{2},z_{3}] \mapsto (\frac{z_{1}}{z_{3}},\frac{z_{2}}{z_{3}})$ on $\mathbb{P}^{2}(\mathbb{C})$.  We first show the result for $L'$. We have for every $\mathfrak{p} \in \mathfrak{B} \cap \mathcal{G}'(\mathfrak{B}) $ : 
$$D\mathcal{G}'_{\mathfrak{p}} - D\mathcal{G}_{ \mathfrak{p}} = D\Pi^{-1}_{  g'(\Pi(\mathfrak{p}))} \cdot Dg'_{\Pi(\mathfrak{p})} \cdot D\Pi_{\mathfrak{p}} - D\Pi^{-1}_{  g(\Pi(\mathfrak{p}))} \cdot D g_{\Pi( \mathfrak{p})} \cdot D\Pi_{\mathfrak{p}} = $$ $$( D\Pi^{-1}_{ g'(\Pi(\mathfrak{p}))} - D\Pi^{-1}_{    g(\Pi(\mathfrak{p}))} ) \cdot D g'_{\Pi(\mathfrak{p})} \cdot D\Pi_{p} +D\Pi^{-1}_{ g(\Pi(\mathfrak{p}))} \cdot (Dg'_{\Pi(\mathfrak{p})}    - Dg_{\Pi(\mathfrak{p})}) \cdot D\Pi_{\mathfrak{p}}                  $$ \newline 
  with : $$D g'_{\Pi(\mathfrak{p})} - D g_{\Pi(\mathfrak{p})} =  \Big( I_{2}+  (DL_{ g(\Pi(\mathfrak{p}))})^{-1} \cdot D(\varepsilon_{1}.h.H)_{g(\Pi(\mathfrak{p}))  } \Big)^{-1}\cdot (DL_{ g(\Pi(\mathfrak{p}))})^{-1} - (DL_{g(\Pi(\mathfrak{p}))})^{-1}  $$ $$ = - (DL_{ g(\Pi(\mathfrak{p}))})^{-1} \cdot  \varepsilon_{1} D(h.H)_{ g(\Pi(\mathfrak{p}))} (DL_{ g(\Pi(\mathfrak{p}))})^{-1} + o(\varepsilon_{1}) $$
$$  D(h.H)_{g(\Pi(\mathfrak{p}))} =     h(g(\Pi(\mathfrak{p})))   \cdot   DH_{g(\Pi(\mathfrak{p}))} +    H(g(\Pi(\mathfrak{p})  )) \cdot Dh_{g(\Pi(\mathfrak{p}))}    $$  Then we have : 
$$Dg'_{\mathfrak{p}} - Dg_{ \mathfrak{p}} =\eta_{1} + \eta_{2} +\eta_{3} +o(\varepsilon_{1})                $$
 where $\eta_{1}= (D\Pi^{-1}_{g'(\Pi(\mathfrak{p}))} - D\Pi^{-1}_{g(\Pi(\mathfrak{p}))} ) \cdot Dg'_{\Pi(\mathfrak{p})} \cdot D\Pi_{\mathfrak{p}}$ and :  $$\eta_{2} =  - D\Pi^{-1}_{g(\Pi(\mathfrak{p}))} \cdot  (DL_{g(\Pi(\mathfrak{p}))})^{-1} \cdot  \varepsilon_{1} \cdot   H(g(\Pi(\mathfrak{p})  )) \cdot  Dh_{g(\Pi(\mathfrak{p}))} \cdot (DL_{g(\Pi(\mathfrak{p}))})^{-1}    \cdot D\Pi_{\mathfrak{p}}$$ $$\eta_{3} =  - D\Pi^{-1}_{ g(\Pi(\mathfrak{p}))} \cdot  (DL_{ g(\Pi(\mathfrak{p}))})^{-1} \cdot  \varepsilon_{1} \cdot h(g(\Pi(\mathfrak{p})))   \cdot   DH_{g(\Pi(\mathfrak{p}))} \cdot  (DL_{ g(\Pi(\mathfrak{p}))})^{-1}   \cdot D\Pi_{\mathfrak{p}}$$

\begin{lem} For any $\mathfrak{p} \in \mathfrak{B} \cap \mathcal{G}'( \mathfrak{B})$ we have : $$\eta_{3}  \in \frac{1}{4 |a|^{2}} \cdot \varepsilon_{1} \cdot |h(\Pi(\mathfrak{p}_{0}))| \cdot V^{j}$$  \end{lem}\begin{proof} This is due to the fact that $H$ has been taken in 4.3.21 so that $-
\frac{h(\Pi(\mathfrak{p}))}{a^{2}} \cdot D(H)_{g(\Pi(\mathfrak{p})) }$ belongs to $\frac{|h(\Pi(\mathfrak{p}_{0}))|}{ |a|^{2}} \cdot  \mathcal{V}_{j}$ and by the definition of $\mathcal{V}_{j}$ (see Lemma \ref{r73}). \end{proof}

\begin{lem} We have :   $ ||\eta_{1}||  <  \frac{1}{4 |a|^{2}} \cdot \varepsilon_{1} \cdot  |h(\Pi(\mathfrak{p}_{0}))| \cdot \min_{1 \le j \le n} r_{j}  $ \end{lem}

\begin{proof}    Since $ 2 \cdot |H|_{\infty}  \le \xi$ (by Lemma \ref{r777}) and by Lemma \ref{r74} for every $\mathfrak{p} \in \mathfrak{B} \cap \mathcal{G}'(\mathfrak{B}) $ we have :

 \begin{eqnarray*} || g'(\Pi(\mathfrak{p}))-  g(\Pi(\mathfrak{p}))|| & \le &  \frac{\varepsilon_{1} \cdot  |h(g(\Pi(\mathfrak{p})))|.\xi}{   \text{inf}_{||w|| = 1} ||D(\Pi \circ aA \circ \Pi^{-1})(w)|| }  \\ &<&  \frac{\varepsilon_{1} \cdot  |h(g(\Pi(\mathfrak{p})))| \cdot \xi}{\sigma' \cdot |a| \cdot   ||A || \cdot ||D\Pi|| \cdot ||D\Pi^{-1}||} \end{eqnarray*}

 Then we have : \begin{eqnarray*}
||\eta_{1}||   &          =  & ||    (D\Pi^{-1}_{ g'(\Pi(\mathfrak{p}))} - D\Pi^{-1}_{   g(\Pi(\mathfrak{p}))} ) \cdot Dg'_{\Pi(\mathfrak{p})} \cdot D\Pi_{\mathfrak{p}}  || \\   &  \le &  ||   D\Pi^{-1}_{g'(\Pi(\mathfrak{p}))} - D\Pi^{-1}_{  g(\Pi(\mathfrak{p}))} || \cdot ||  Dg'_{\Pi(\mathfrak{p})}  || \cdot || D\Pi || \\ 
   & <  &  \frac{\varepsilon_{1} \cdot  |h(g(\Pi(\mathfrak{p})))| \cdot \xi}{\sigma' \cdot |a|  ||A || \cdot ||D\Pi|| \cdot ||D\Pi^{-1}||}   \cdot  ||\Pi^{-1}||_{C^{2}} \cdot ||D g'_{\Pi(\mathfrak{p})} || \cdot || D\Pi || \\ 
   & <  &    \frac{\varepsilon_{1} \cdot 2|h(\Pi(\mathfrak{p}_{0}))| \cdot \xi}{\sigma' \cdot  |a| \cdot  ||A || \cdot ||D\Pi|| \cdot ||D\Pi^{-1}||}     \cdot ||\Pi^{-1}||_{C^{2}} \cdot( \frac{1}{  |a| \cdot   ||A ||} \cdot  ||D\Pi||  \cdot ||D\Pi^{-1}||) \cdot || D\Pi || \\  & < & \frac{1}{4 |a|^{2}} \cdot \varepsilon_{1} \cdot  |h(\Pi(\mathfrak{p}_{0}))| \cdot \text{min}_{1 \le j \le n} r_{j}   \end{eqnarray*} by the inequality on $\xi$ (see 4.2.9). \end{proof}

\begin{lem} We have : $||  \eta_{2} ||< \frac{1}{4 |a|^{2}}.\varepsilon_{1}. |h(\Pi(\mathfrak{p}_{0}))|.\min_{1 \le j \le n} r_{j} $,
\end{lem}
\begin{proof} \begin{eqnarray*}    || \eta_{2}  ||  \text{    }          &  \le &   ||D\Pi^{-1}|| \cdot ||    DL^{-1}   || \cdot \varepsilon_{1} \cdot ||H|| \cdot ||Dh|| \cdot   ||    DL^{-1}   || \cdot ||D\Pi|| \\ & \le &    ||D\Pi^{-1}|| \cdot     \Big( \frac{1}{ |a| \cdot ||A ||} \cdot  ||D\Pi|| \cdot  ||D\Pi^{-1}|| \Big) ^{2}  \cdot   \varepsilon_{1} \cdot \xi \cdot ||Dh||    \cdot   ||D\Pi||    \\
  & = & \frac{1}{ |a|^{2}} \cdot \varepsilon_{1} \cdot \xi \cdot   \frac{1}{||A ||^{2}} \cdot ||D\Pi| |^{3}   \cdot  |   |D\Pi^{-1}|| ^{3} \cdot || Dh||  \\ & \le & \frac{1}{4 |a|^{2}} \cdot \varepsilon_{1} \cdot  |h(\Pi(\mathfrak{p}_{0}))| \cdot \text{min}_{1 \le j \le n} r_{j}  \end{eqnarray*}
by the inequality on $\xi$ (see 4.2.9) and $|| Dh(z) || \le \sigma \cdot |h(z)| \le \sigma \cdot 2 |h(\Pi(\mathfrak{p}_{0}))|$ for $z \in \mathfrak{B}$ by inequality (1) of Lemma \ref{r75}.  \qedhere
\end{proof}

The three previous lemmas imply that on $\mathfrak{B} \cap \mathcal{G}'(\mathfrak{B})$, $D\mathcal{G}'_{\mathfrak{p}} - D\mathcal{G}_{ \mathfrak{p}}$ belongs to $t \cdot V^{j}$ with $t = \frac{\varepsilon_{1} \cdot |h(\Pi(\mathfrak{p}_{0}  ))|}{ |a|^{2}}$. Then by continuity, for a given $\varepsilon_{1}$ (and then a given $L'$), there exists a neighborhood $\mathcal{X}_{\varepsilon_{1}}$ of $L'$ in $\text{Hol}_{d'}$ such that for every sufficiently small perturbation $L''' \in \mathcal{X}_{\varepsilon_{1}}$ of $L'$, if $j$ is such that $\mathcal{G}'''_{j}(\mathfrak{B}) \subset \mathfrak{B}_{p}   $ then $\mathcal{G}_{j}'''  $ is quasi-linear of type $(t\varepsilon_{1},p)$. The proof of Proposition \ref{r71} is complete.  \qedhere

\end{proof}

\subsection{Well oriented postcritical set}
\begin{nt5} We fix $\mathfrak{pc}$ a point of $\Pi^{-1}(p_{c})$ (remind that the periodic point $p_{c}$ was defined in Proposition \ref{r41}) . \end{nt5}

\begin{nt5}
We denote by $\mathrm{PCrit}(L)$ the postcritical set of $L$, this is the set $ \mathrm{PCrit}(L) = \bigcup_{n \ge 0} (L)^{n}(\mathrm{Crit}(L))$ where $\mathrm{Crit}(L)$ is the critical set of $L$. The notation will be the same for perturbations $L',L'',L''$.   
\end{nt5}
We pick a vector $w_{1}$ and a value $\theta$ satisfying Property (P) of Corollary \ref{r55}. Still according to Corollary \ref{r55}, there exists an open set of admissible values for $w_{1}$ so we choose to take it in the following way. The map $\mathcal{L}^{n_{pc} \cdot \text{ord}(A)}$ is an affine map on the torus $\mathbb{T}$ of linear part $a^{  n_{pc} \cdot \text{ord}(A) } \cdot A^{n_{pc} \cdot \text{ord}(A)} =a^{n_{pc} \cdot \text{ord}(A)} \cdot  I_{2}^{n_{pc}} = a^{n_{pc} \cdot \text{ord}(A)} \cdot I_{2}$ with $|a| \ge 2$ (see 3.2.12). Points with dense forward orbit for $\mathcal{L}^{n_{pc} \cdot \text{ord}(A)}$ are dense in $\mathbb{T}$. Moreover, since $n_{pc}$ divides $n_{pc} \cdot \text{ord}(A)$, $\mathfrak{pc}$ is a fixed point of $\mathcal{L}^{n_{pc} \cdot \text{ord}(A)}$. We pick $w_{1}$ such that $\mathfrak{pc}+\pi(w_{1})$ is a point of dense forward orbit for $\mathcal{L}^{n_{pc} \cdot \text{ord}(A)}$ (remind that $\pi : \mathbb{C}^{2} \rightarrow \mathbb{T}$ is the natural projection). Since the linear part of $\mathcal{L}^{n_{pc} \cdot \text{ord}(A)}$ is $ a^{n_{pc} \cdot \text{ord}(A)} \cdot I_{2}$ and $\mathfrak{pc}$ is a fixed point of $\mathcal{L}^{n_{pc} \cdot \text{ord}(A)}$, we have that the whole forward orbit of $\mathfrak{pc}+\pi(w_{1})$ under $\mathcal{L}^{n_{pc} \cdot \text{ord}(A)}$ is contained in the line going through $\mathfrak{pc}$ and $\mathfrak{pc}+\pi(w_{1})$. In particular, this line is dense in the torus $\mathbb{T}$. We pick $w_{2}$ such that $(w_{1},w_{2})$ is a basis of $\mathbb{C}^{2}$ and $\pi(w_{2})$ is not tangent to $\Pi^{-1}(\mathrm{PCrit}(L))$ at $\mathfrak{pc}$. \newline \newline
Here is the main result of this subsection :

\begin{pr} \label{r97}Let $\mathfrak{B}$ be as in 4.3.20. There exists a neighborhood $\mathcal{W}(L)$ of $L$ in $\mathrm{Hol}_{d'}$ such that : every map $L'=L'_{\varepsilon_{1}}$ as in Proposition \ref{r71777} is accumulated by maps $L'''_{\varepsilon_{1},\varepsilon_{2},\varepsilon_{3}}=L'''$ in $\mathcal{W}(L)$ such that there exists a component $\Gamma \subset \Pi^{-1}(    \mathrm{PCrit}(L''')   )$ whose restriction to $\mathfrak{B}$ is a $(\theta,w_{1})$-quasi-diameter (remind that this notion was defined in Definition \ref{def33}).
\end{pr}

The following lemma is well known.
\begin{lem} \label{r98}
Let $\mathfrak{L}$ be a linear automorphism of $\mathbb{C}^{2}$ and $\Gamma \subset \mathbb{C}^{2}$ a complex submanifold through 0 such that : \begin{enumerate}
\item   the eigenvalues $\lambda,\mu$ of $\mathfrak{L}$ are such that $|\lambda| > |\mu|>1$. Let $w_{\lambda}$ and $w_{\mu}$ be the respective eigenvectors.
\item $w_{\mu}$ is transverse to $\Gamma$ at 0 \end{enumerate}
Then, $(\mathfrak{L}^{k} (\Gamma))_{k \ge 0}$ converges uniformly to the line $\mathbb{C} \cdot w_{\lambda}$ in the $C^{1}$-topology.

\begin{proof}
We can take $w_{\lambda} = e_{1}$ and $w_{\mu} = e_{2}$. The eigenvector $w_{\mu}$ of $\mu$ is transverse to $\Gamma$ at 0. Then locally $\Gamma$ is a graph $\gamma$ over a small disk $\mathbb{D}_{\gamma} \subset \mathbb{D}$ : $\{(t, \gamma(t)) : t \in \mathbb{D}_{\gamma})   \} $. For every $k \in \mathbb{N}$, $\mathfrak{L}^{k}(   \{(t, \gamma(t)) : t \in \mathbb{D}_{\gamma})   \} ) = \{(\lambda^{k} \cdot t , \mu^{k} \cdot \gamma(t)) : t \in \mathbb{D}_{\gamma} \}$. Since $|\lambda| >1$, for large $k$, we have $\mathbb{D} \subset \lambda^{k} \cdot \mathbb{D}_{\gamma} $. Then  $\mathfrak{L}^{k}(   \{(t, \gamma(t)) : t \in \mathbb{D}_{\gamma})   \} ) $ contains $\{(s, \mu^{k} \cdot \gamma(\frac{s}{\lambda^{k}})) : s \in \mathbb{D}\}$. But there exists $C_{\gamma}>0$ such that $|\gamma(t) | < C_{\gamma} \cdot t$ near 0. Then $\mu^{k} \cdot \gamma(\frac{s}{\lambda^{k}}) < C_{\gamma} \cdot (\frac{\mu}{\lambda})^{k}$ converges uniformly to 0 on $\mathbb{D}_{\gamma}$. Then, for every $\theta'>0$, there exists $k$ such that $\mathfrak{L}^{k}(   \{(t, \gamma(t)) : t \in \mathbb{D}_{\gamma})   \} ) $ contains $ \{ (s, \tilde{\gamma}(s)): s \in \mathbb{D} \} =  \{  (s, \mu^{k} \cdot \gamma(\frac{s}{\lambda^{k}} ) : s \in \mathbb{D}\} $ with $   | \tilde{\gamma}(s)  | \le \theta'  $. Then by the Cauchy inequality this implies that $|(\tilde{\gamma})'(s) | \le \theta'$.  \qedhere

\end{proof}

\end{lem}

\begin{nt5}
We denote for every $(\lambda,\mu ) \in (\mathbb{C}^{*})^{2}$ by $\mathrm{Diag}_{\lambda,\mu} $ the following map from $\mathbb{C}^{2} $ to $\mathbb{C}^{2}$ :
$$ \mathrm{Diag}_{\lambda,\mu} : (z_{1},z_{2}) \mapsto (\lambda \cdot z_{1}, \mu \cdot z_{2}) $$

\end{nt5}

\begin{lem}\label{igo}
The linear part $A \in G_{\text{Latt\`es}}$ of $L$ is diagonalizable. \end{lem}
\begin{proof}
Since $G_{\text{Latt\`es}}$ is a finite group, we have that $A$ is of finite order. In particular, $A^{\text{ord}(G_{\text{Latt\`es}})} = I_{2}$. Then $R( A) = 0$ where $R(X)=X^{\text{ord}(G_{\text{Latt\`es}})}-1$ has simple roots. Then $A$ is diagonalizable.  \qedhere
\end{proof}

\begin{lem} \label{r99} Let $(f_{\varepsilon })_{\varepsilon \in \mathbb{D}^{3}}$ be a holomorphic family of holomorphic germs defined in a neighborhood $\mathcal{U}$ of 0 such that for every $\varepsilon \in \mathbb{D}^{3}$, $D(f_{\varepsilon})_{0}$ is diagonalizable and 0 is a repelling fixed point for $f_{\varepsilon}$. We denote by $\lambda(\varepsilon)$, $\mu(\varepsilon)$ the eigenvalues of $D(f_{\varepsilon})_{0}$ and by $w_{\lambda}$, $w_{\mu}$ the associated eigenvectors. We suppose that in the family $(f_{\varepsilon})_{\varepsilon \in \mathbb{D}^{3}}$, $w_{\lambda} = w_{1}$ and $w_{\mu} = w_{2}$ are constant. We suppose that $|\lambda(0)|^{2} >|\mu (0)| \ge  |\lambda(0) |$. Then there exists a neighborhood $\mathcal{U}' \subset \mathcal{U}$ of 0 and a neighborhood $\mathcal{V}$ of 0 in $\mathbb{D}^{3}$ such that for every $\varepsilon \in \mathcal{V}$, $f_{\varepsilon}$ is holomorphically linearizable in $\mathcal{U}'$ : there exists a holomorphic map $\varphi_{f_{\varepsilon}}$ defined on $\mathcal{U'}$ such that : $$ \mathrm{Diag}_{\lambda(\varepsilon),\mu(\varepsilon)} \circ \varphi_{f_{\varepsilon}} = \varphi_{f_{\varepsilon}} \circ f_{\varepsilon} $$
Moreover, $\varphi_{f_{\varepsilon}}$ varies continuously with $\varepsilon$ in the $C^{0}$ topology.
\end{lem}

The proof will be based on the following well known result (Theorem 6.2.3 in \cite{jp7}).

\begin{pr}    Let $F$ be an invertible map with repulsive fixed point 0. Suppose that the eigenvalues $\lambda,\mu$ of $DF_{0}$ satisfy the condition $ |\lambda|^{2} >|\mu | \ge  |\lambda |>1$. Then $F$ is holomorphically conjugate to $\mathrm{Diag}_{\lambda,\mu} $.
\end{pr}
The following lemma will be used to compare $\varphi_{f_{\epsilon}}$ and $\Pi$ : 
\begin{lem3} \label{rlin} Let $F$ be an invertible map in a neighborhood of 0 with a repelling fixed point at 0. Let us denote the eigenvalues of  $DF_{0}$ by $\lambda,\mu$. Let us suppose that  $\varphi_{1}$ and $\varphi_{2}$ are two holomorphic maps conjugating $F$ to $\mathrm{Diag}_{\lambda,\mu} $. Then $\varphi_{1} \circ \varphi_{2}^{-1}$ is linear.

\end{lem3}
\begin{proof}
Let us write $\chi = \varphi_{1} \circ \varphi_{2}^{-1}=(\chi^{1},\chi^{2})$ and $\chi^{j}(z)= \sum_{k \ge 1}\chi^{j}_{k} \cdot  z^{k}$ where $z^{k} = z_{1}^{k_{1}} \cdot z_{2}^{k_{2}}$ and $j \in \{1,2\}$. We have that $\chi$ commutes with $\mathrm{Diag}_{\lambda,\mu}$. Then : $$\lambda \cdot  \sum_{|k| \ge 1}\chi^{1}_{k} \cdot  z^{k} = \sum_{|k| \ge 1}\chi^{1}_{k} \cdot (\text{Diag}_{\lambda,\mu}  (z))^{k} \text{   and   }    \mu \cdot  \sum_{|k| \ge 1}\chi^{2}_{k} \cdot  z^{k} = \sum_{|k| \ge 1}\chi^{2}_{k} \cdot (\mathrm{Diag}_{\lambda,\mu} (z))^{k}      $$ In particular, since $ \lambda , \mu \neq 1$ this implies that $\chi^{1}_{k}=\chi^{2}_{k} = 0$ for every $|k| >1$. \qedhere 
 
\end{proof}  We now prove Lemma \ref{r99}.
\begin{proof}[Proof of Lemma \ref{r99}] We take a neighborhood $\mathcal{V}$ of 0 in $\mathbb{D}^{3}$ such that for every $\varepsilon \in \mathcal{V}$ we have that $|\lambda(\varepsilon)|^{2}<|\mu(\varepsilon)|$ and $|\mu(\varepsilon)|^{2}<|\lambda(\varepsilon)|$. It is a consequence of Theorem 6.2.3 of \cite{jp7} (this result goes back to Poincar\'e) that for every $\varepsilon \in \mathcal{V}$, $f_{\varepsilon}$ is holomorphically linearizable at 0 in some neighborhood $\mathcal{U}'_{\varepsilon}$ of 0. We show here that the linearizing map $\varphi_{f_{\varepsilon}}$ varies continuously with $\varepsilon$ in the $C^{0}$ topology. This will imply in particular that the neighborhood $\mathcal{U}'_{\varepsilon}$ can be taken uniform $\mathcal{U}'$ in $\varepsilon$. \newline \newline For this we follow the proof of Theorem 6.2.3 of \cite{jp7} and its notations (we just replace the $A$ of the original proof by $C$ to avoid confusion with the linear part of the Latt\`es map). The proof is divided into 3 steps. \newline \newline The first step itself is divided into two steps. The first one is a linear change of coordinates that we will denote by $\varphi_{lin}$ which locally conjugates $f_{\epsilon}$ to $(z_{1},z_{2}) \rightarrow (\lambda \cdot z_{1}+ \ldots , \mu \cdot z_{2} + \nu \cdot z_{1} + \ldots ) = (\lambda \cdot z_{1}+ C(z) , \mu \cdot z_{2} + \nu \cdot z_{1} + \ldots )$. $\varphi_{lin}$ is not unique but it becomes unique if $w_{1}$ is sent on $e_{1}$ and $w_{2}$ is sent on $e_{2}$. Thus this map $\varphi_{lin} = \varphi_{lin}(f_{\varepsilon})$ is uniquely defined and varies continuously in the $C^{0}$ topology. The second one is a change of coordinates that we will denote $\varphi_{1}(z) = (\varphi^{1}_{1}(z_{1}),z_{2})$ such that $\varphi^{1}_{1}(z_{1}) = z_{1}+ \sum_{k = 0}^{+ \infty} \frac{1}{\lambda^{k+1}} \cdot C(f_{\varepsilon}^{k}(z_{1}))$. There exists a constant $K$ such that $ |C(z_{1}) | \le K  |z_{1} |^{2}$ for every map $f_{\varepsilon}$ with $\varepsilon \in \mathcal{V} $ (reducing $\mathcal{V}$ if necessary). Since $\varphi^{1}_{1}$ is the sum of a normally convergent series whose terms all vary continuously, $\varphi^{1}_{1}$ and then $\varphi_{1} \circ \varphi_{lin}$ vary continuously in the $C^{0}$ topology. After these two changes of coordinates, $f_{\varepsilon}$ is reduced to the form $(z_{1},z_{2}) \mapsto (\lambda \cdot z_{1},g(z_{1},z_{2}))$ with $g$ varying continuously. 
\newline \newline In the second step, one defines some infinite product $\gamma(z) =  \prod_{k = 0}^{+ \infty} (1+B(F^{n}(z))$ where $\frac{\partial g}{\partial z_{2}} (z) = \mu(1+B(z))$. We have $ |B(z) | \le K'  |z |^{2}$ and reducing $\mathcal{V}$ if necessary, we can suppose this estimate is true for every $\varepsilon \in \mathcal{V}$. Then $\gamma$ is normally convergent and varies continuously. The map $\psi$ such that $\frac{\partial \psi}{\partial z_{2}} = \gamma$ then still varies continuously, just as $\varphi_{2}(z) = (z_{1},\psi(z))$. After this third change of coordinates, $f_{\varepsilon}$ is reduced to the form $(z_{1},z_{2}) \mapsto (\lambda \cdot z_{1}, \mu \cdot z_{2} + h(z_{1})  )$ with $h$ varying continuously. \newline \newline Finally, the last change of coordinates $\varphi_{3}$ is of the form $(z_{1},\eta(z))$ with $\eta(z) = z_{2}+ q(z_{1})$ and $q(z_{1} ) =q_{1} \cdot z_{1} +q_{2} \cdot z_{1}^{2}+...$ with $q_{j} = \frac{h_{j}}{\mu-\lambda^{j}}$ for each $j \ge 2$. Since $h$ varies continuously, so do the coefficients $q_{j}$ for $j \ge 2$. When $\lambda = \lambda(\epsilon) \neq \mu(\epsilon) = \mu$, we set $q_{1} = \frac{h_{1}}{\mu-\lambda}$. If $\lambda = \lambda(\epsilon) = \mu(\epsilon) = \mu$, we have $h_{1} = 0$ because $D(f_{\varepsilon})_{0}$ is diagonalizable (see \cite{jp7}, p213). Then we can extend $q_{1}$ by continuity and by Lemma \ref{rlin}, one gets the only value of $q_{1}$ for which $w_{1}$ is sent on $e_{1}$ and $w_{2}$ on $e_{2}$ under $D(\varphi_{3} \circ \varphi_{2} \circ \varphi_{1} \circ \varphi_{lin})_{0}$. Finally, $q_{1}$, $q$, $\eta$ and $\varphi_{3}$ vary continuously. Setting $\varphi_{f_{\epsilon}} = \varphi_{3} \circ \varphi_{2} \circ \varphi_{1} \circ \varphi_{lin}$, $f_{\varepsilon}$ is holomorphically linearizable by $\varphi_{f_{\varepsilon}}$ for every $\varepsilon$ (because $D(f_{\varepsilon})_{0}$ is diagonalizable) and $\varphi_{f_{\varepsilon}}$ varies continuously in the $C^{0}$ topology. 
\end{proof}

Remind that $c$ is a point of $\mathbb{P}^{2}(\mathbb{C})$ which was defined in 4.3.17 and that the notation $p(L'')$ was introduced in Notation \ref{notf}. In the following lemma, we perturb $L'=L'_{\varepsilon_{1}}$ into $L''=L''_{\varepsilon_{1},\varepsilon_{2}}$ to ensure that the critical point $c$ is not singular.

\begin{lem} \label{r933} There exists a holomorphic family of holomorphic maps of $\mathbb{P}^{2}(\mathbb{C})$ denoted by  $(L''_{\varepsilon_{1},\varepsilon_{2}})_{(\varepsilon_{1},\varepsilon_{2}) \in \mathbb{D}^{2}}$ such that :

\begin{enumerate}
\item for every $\varepsilon_{1} \in \mathbb{D}$, $L''_{\varepsilon_{1},0} = L'_{\varepsilon_{1}}$
\item $p_{c}$ is in the postcritical set of $L''_{\varepsilon_{1},\varepsilon_{2}}$ and \emph{the postcritical set is not singular at $p_{c}$ for $\varepsilon_{2} \neq 0$}
\item  $p(L''_{\varepsilon_{1},\varepsilon_{2}}) =p_{c}$ is periodic for $L''_{\varepsilon_{1},\varepsilon_{2}}$ and is in the forward orbit of $c$ : $p_{c} = (L''_{\varepsilon_{1},\varepsilon_{2}})^{n_{c}}(c)$ and $(L''_{\varepsilon_{1},\varepsilon_{2}})^{n_{pc}}(p_{c}) = p_{c}$
\item  $D( (L''_{\varepsilon_{1},\varepsilon_{2}})^{n_{pc}})_{p_{c}} = D(L^{n_{pc}})_{p_{c}}$

\end{enumerate}

\end{lem}
\begin{proof}  We first make an invertible linear change of coordinates so that in the new coordinates $[x_{1},x_{2},x_{3}]$, the point $c$ is equal to $[0,0,1]$ and the point $L(c)$ is in the chart $\{x_{3} \neq 0\}$. We choose homogenous polynomials of degree 1 in the variables $x_{1},x_{2},x_{3}$ denoted by $R_{2}, \ldots ,R_{n_{c}+n_{pc}-1}$ such that :  \begin{equation} \label{eq799} R_{2}(L(c)) =  \cdots = R_{n_{c}-1}(L^{n_{c}-1}(c)) = R_{n_{c}}(p_{c}) = \cdots = R_{n_{c}+n_{pc}-1}(L^{n_{pc}-1}(p_{c})) = 0 \end{equation} \begin{equation} \label{eq795} R_{2}(c)  \neq 0 , \cdots ,  R_{n_{c}-1}(c) \neq 0 \text{  ,  } R_{n_{c}}(c) \neq 0 , \cdots, R_{n_{c}+n_{pc}-1}(c)  \neq 0 \end{equation} This is possible since $c, \ldots ,L^{n_{pc}-1}(p_{c})$ are distinct. We denote :  $$\gamma_{1}(x_{1}  ,x_{2}) =   \prod _{2 \le k \le n_{c}+n_{pc}-1}  (R_{k}(x_{1},x_{2},1 ))^{2} $$ In $\{x_{3} \neq 0\}$, the critical set of $L'$ is the set $\{  \text{Jac}(P')  =     0 \}$ where $\text{Jac}(P')$ is equal to $\frac{\partial P'_{1}}{\partial x_{1}} \cdot \frac{\partial P'_{2}}{\partial x_{2}}- \frac{\partial P'_{1}}{\partial x_{2}} \cdot \frac{\partial P'_{2}}{\partial x_{1}}$. The critical set at $c$ is not singular if the gradient of the map $(x_{1},x_{2}) \mapsto {\frac{\partial P'_{1}}{\partial x_{1}} }\cdot \frac{\partial P'_{2}}{\partial x_{2}}- \frac{\partial P'_{1}}{\partial x_{2}} \cdot \frac{\partial P'_{2}}{\partial x_{1}}  $ is non zero at $c$, in particular if : $$\frac{ \partial} {\partial x_{1}} ( \text{Jac}(P')) (c) =  {\frac{\partial^{2} P'_{1}}{\partial x_{1}^{2}} \cdot \frac{\partial P'_{2}}{\partial x_{2}} +\frac{\partial P'_{1}}{\partial x_{1}} \cdot \frac{\partial^{2} P'_{2}}{\partial x_{1} \partial x_{2} } - \frac{\partial^{2} P'_{1}}{\partial x_{1} \partial x_{2}} \cdot \frac{\partial P'_{2}}{\partial x_{1}}  - \frac{\partial P'_{1}}{\partial x_{2}} \cdot \frac{\partial^{2} P'_{2}}{\partial x_{1}^{2}}} \neq 0$$  
If this is the case, there is nothing to do and we can take $L'' = L'$. Let us suppose this is not so. We distinguish two cases. \newline \newline  First case : we suppose that $\frac{\partial P'_{2}}{\partial x_{2}} \neq 0$. For every $\varepsilon_{2} \in \mathbb{C}$ we consider the following perturbation of $L'$ defined by $L''_{\varepsilon_{1},\varepsilon_{2}}= L'' = (\frac{ P''_{1}}{P''_{3}} , \frac{P''_{2}}{P''_{3}})$ with $P''_{2} = P'_{2}$, $P''_{3} = P'_{3}$ and : 
$$P''_{1}(x_{1},x_{2}) = P'_{1}(x_{1},x_{2})+ \varepsilon_{2} \cdot  \gamma_{1}(x_{1},x_{2}) \cdot x_{1}^{2}$$
Because of the choice of the degree $d$ in 4.2.13 and 4.2.14, we have $\text{deg}(P''_{1}) \le \text{deg}(P'_{1})$. Since the property of being a holomorphic mapping is open, $L''$ is a holomorphic mapping on $\mathbb{P}^{2}(\mathbb{C})$ for sufficiently small values of $\varepsilon_{2}$. Then, item 1 is obvious. Because of the quadratic terms $R_{k}^{2}$ in the definition of $\gamma_{1}$, $c$ stays preperiodic (with the same periodic orbit $p(L''_{\varepsilon_{1},\varepsilon_{2}}) =p_{c}$) for $L''_{\varepsilon_{1},\varepsilon_{2}}$, this implies item 3. Still because of the quadratic terms $R_{k}^{2}$ in $\gamma_{1}$ we have that $D_{L(c)}(L''_{\varepsilon_{1},\varepsilon_{2}})=D_{L(c)}(L'_{\varepsilon_{1}}) = D_{L(c)} L$, $\cdots$ , $D_{(L)^{n_{pc}-1}(p_{c})} (L''_{\varepsilon_{1},\varepsilon_{2}})  =D_{(L)^{n_{pc}-1}(p_{c})} (L'_{\varepsilon_{1}}) =  D_{(L)^{n_{pc}-1}(p_{c})}L$, so we both have that $p_{c}$ is in the postcritical set of $L''_{\varepsilon_{1},\varepsilon_{2}}$ and that item 4 is true. Moreover we have $D_{c}L'' = D_{c}L$ so $c$ is still critical. The only second order partial derivative which depends on $\varepsilon_{2}$ is : $\frac{\partial^{2} P''_{1}}{\partial x_{1}^{2}}(c) = \frac{\partial^{2} P'_{1}}{\partial x_{1}^{2}}(c) + 2 \cdot \gamma_{1}(c) \cdot \varepsilon_{2} $ with $\gamma_{1}(c) \neq 0$. Then the map $\varepsilon_{2} \mapsto \frac{ \partial} {\partial x_{1}} ( \text{Jac}(P''))(c)$ is an affine map in $\varepsilon_{2}$ of non zero coefficient equal to  $2 \cdot \gamma_{1}(c) \cdot \frac{\partial P'_{2}}{\partial x_{2}}$. Then, it is non zero for $\varepsilon_{2} \in \mathbb{D}^{*}$. This implies that the critical set is not singular at $c$. Then there is a component of the critical set at $c$ which is not singular. Since $DL_{c}$, $\cdots$ ,$DL_{L^{n_{pc}-1}(p_{c})}$ are not singular, there is a component of the postcritical set at $p_{c}$ which is not singular. Thus item 2 is true. \newline \newline 
Second case :  we suppose that $\frac{\partial P'_{2}}{\partial x_{2}} = 0$. For every $\varepsilon_{2} \in \mathbb{C}$ we consider the following perturbation of $L'$ defined by $L''_{\varepsilon_{1},\varepsilon_{2}}= L'' = (\frac{ P''_{1}}{P''_{3}} , \frac{P''_{2}}{P''_{3}})$ with $P''_{3} = P'_{3}$ and : 
$$P''_{1}(x_{1},x_{2}) = P'_{1}(x_{1},x_{2})+ \varepsilon_{2} \cdot \gamma_{1}(x_{1},x_{2}) \cdot x_{1}$$
$$P''_{2}(x_{1},x_{2}) = P'_{2}(x_{1},x_{2})+ \varepsilon_{2} \cdot \gamma_{1}(x_{1},x_{2}) \cdot x_{1}x_{2}$$
Because of the choice of the degree $d$ in 4.2.13 and 4.2.14, we have $\text{deg}(P''_{1}) \le \text{deg}(P'_{1})$ and $\text{deg}(P''_{2}) \le \text{deg}(P'_{2})$. Since the property of being a holomorphic mapping is open, $L''$ is a holomorphic mapping for sufficiently small values of $\epsilon_{2}$. Then, item 1 is obvious. As in the first case, items 3 and 4 are true and $p_{c}$ stays postcritical. We have : 
$$D_{c}L'' = D_{c}L'+ \varepsilon_{2}  \cdot \gamma_{1}(c) \cdot  \begin{pmatrix} 1 & 0 \\  0 & 0 \end{pmatrix}$$  Since at the point $c$, we have both $\frac{\partial P'_{2}}{\partial x_{2}} = 0$ and $  \text{Jac}(P') (c)  = \frac{\partial P'_{1}}{\partial x_{1}} \cdot \frac{\partial P'_{2}}{\partial x_{2}}- \frac{\partial P'_{1}}{\partial x_{2}} \cdot \frac{\partial P'_{2}}{\partial x_{1}}=0$, this implies that we still have $\text{Jac}(P'') (c) = 0$ and the point $c$ is still critical. The only second order partial derivative which depends on $\varepsilon_{2}$ is : $\frac{\partial^{2} P''_{2}}{\partial x_{1} \partial x_{2}}(c) = \frac{\partial^{2} P'_{2}}{\partial x_{1} \partial x_{2}}(c) +  \gamma_{1}(c) \cdot \varepsilon_{2} $ with $\gamma_{1}(c) \neq 0$. Then the map $\varepsilon_{2} \mapsto \frac{ \partial} {\partial x_{1}} ( \text{Jac}(P''))(c)$ is a polynomial of degree 2 in $\varepsilon_{2}$ of non zero coefficient of degree 2 equal to $(\gamma_{1}(c))^{2}$. Then, rescaling if necessary, it is non zero for $\varepsilon_{2} \in \mathbb{D}^{*}$. As in case 1, we conclude that item 2 is satisfied. This concludes the proof of the proposition.  \qedhere

\end{proof}

Remind that $w_{1}$ and $w_{2}$ were defined just at the beginning of this subsection. The notation $p(L'')$ was introduced in Notation \ref{notf}. In the following lemma, we perturb the periodic orbit $p_{c}$ in such a way that we can choose the two eigenvalues at this periodic point .

\begin{lem} \label{r93} There exists a holomorphic family of holomorphic maps of $\mathbb{P}^{2}(\mathbb{C})$ denoted by  $(L'''_{\varepsilon_{1},\varepsilon_{2},\varepsilon_{3}})_{(\varepsilon_{1},\varepsilon_{2},\varepsilon_{3}) \in \mathbb{D}^{3}}$ such that :

\begin{enumerate}

\item for every $\varepsilon_{1},\varepsilon_{2} \in \mathbb{D}$, $L'''_{\varepsilon_{1},\varepsilon_{2},0} = L''_{\varepsilon_{1},\varepsilon_{2}}$
\item  $p(L''') =p_{c}$ is periodic for $L'''$ ($(L''')^{n_{pc}}(p_{c}) = p_{c}$) and is in the postcritical set of $L'''$ ($p_{c} = (L''')^{n_{c}}(c)$)
\item   if $\varepsilon_{3} > 0$, then the eigenvalues $\lambda,\mu$ of $D(L''')^{n_{p_{c}}}$ at $p_{c}$ are such that : $|\mu|^{2}>|\lambda| > |\mu|$ \item  the eigenvector $w_{\mu}$ associated to $\mu$ at $p_{c}$ is equal to $D\Pi_{\mathfrak{pc}}(w_{2})$ and then transverse to the postcritical set at $p_{c}$ 
\item the eigenvector $w_{\lambda}$ associated to $\lambda$ at $p_{c}$ is equal to $D\Pi_{\mathfrak{pc}}(w_{1})$

\end{enumerate}

\end{lem}
\begin{proof} 
We first make an invertible linear change of coordinates so that in the new coordinates $[y_{1},y_{2},y_{3}]$, the point $p_{c}$ is equal to $[0,0,1]$ and the point $L(p_{c})$ is in the chart $\{y_{3} \neq 0\}$. We choose homogenous polynomials of degree 1 in the variables $y_{1},y_{2},y_{3}$ denoted by $S_{1}, \ldots ,S_{n_{c}-1},S_{n_{c}+1}, \dots ,S_{n_{c}+n_{pc}-1}$ such that :  \begin{equation} \label{eq73} S_{1}(c) = S_{2}(L(c))= \cdots = S_{n_{c}-1}(L^{n_{c}-1}(c)) =  S_{n_{c}+1}(L(p_{c})) = \cdots = S_{n_{c}+n_{pc}-1}(L^{n_{pc}-1}(p_{c})) = 0 \end{equation} \begin{equation} \label{eq75} S_{1}(p_{c})  \neq 0 \text{ , }S_{2}(p_{c})  \neq 0, \cdots, S_{n_{c}-1}(p_{c})  \neq 0 \text{ , }  S_{n_{c}+1}(p_{c})  \neq 0, \cdots ,S_{n_{c}+n_{pc}-1}(p_{c}) \neq 0 \end{equation} This is possible since $c, \ldots,L^{n_{pc}-1}(p_{c})$ are distinct. We denote :  $$\gamma_{2}(y_{1},y_{2}) =     \prod _{ j \neq n_{c}}  (S_{j}(y_{1},y_{2},1))^{2}$$ For every $\varepsilon_{3},\kappa_{i} \in \mathbb{C}$ we consider the following perturbation of $L''$ defined by $L'''_{\varepsilon_{1},\varepsilon_{2},\varepsilon_{3}}= L''' = (\frac{ P'''_{1}}{P'''_{3}} , \frac{P'''_{2}}{P'''_{3}}) $ with $P'''_{3} = P''_{3}$ and : 
$$P'''_{1}(y_{1},y_{2}) = P''_{1}(y_{1},y_{2})+ \varepsilon_{3} \cdot \gamma_{2}(y_{1},y_{2}) \cdot (\kappa_{1}y_{1}+\kappa_{2}y_{2})$$
$$P'''_{2}(y_{1},y_{2}) = P''_{2}(y_{1},y_{2})+ \varepsilon_{3} \cdot \gamma_{2}(y_{1},y_{2}) \cdot (\kappa_{3}y_{1}+\kappa_{4}y_{2})$$ We are going to choose carefully the coefficients $\kappa_{i}$ in order to control the differential $D((L''')^{n_{p_{c}}})_{p_{c}}$. Because of the choice of the degree $d$ in 4.2.13 and 4.2.14, we have $\text{deg}(P'''_{1}) \le \text{deg}(P''_{1})$ and $\text{deg}(P'''_{2}) \le \text{deg}(P''_{2})$. Since the property of being a holomorphic mapping is open, $L'''$ is a holomorphic mapping on $\mathbb{P}^{2}(\mathbb{C})$ for sufficiently small values of $\varepsilon_{3}$. Then, item 1 is obvious. Then, because of the quadratic terms $S_{j}^{2}$ in $\gamma_{2}$, it is clear that $c$ stays preperiodic (with the same periodic orbit $p_{c}$) for $L'''$ and $D_{L(p_{c})}L''' = D_{L(p_{c}) }L$, $\cdots$, $D_{L^{n_{pc}-1}(p_{c})}L''' = D_{L^{n_{pc}-1}(p_{c}) }L$. This shows item 2.  In the chart $\{y_{3} \neq 0\}$, we have : $$D_{p_{c}}L''' = D_{p_{c}}L''+ \varepsilon_{3}  \cdot \gamma_{2}(p_{c}) \cdot  \begin{pmatrix} \kappa_{1} & \kappa_{2} \\  \kappa_{3} & \kappa_{4} \end{pmatrix}$$ with $ \gamma_{2}(p_{c}) \neq 0 $. We have $D_{L(p_{c})}L''' = D_{L(p_{c}) }L$, $\cdots$ , $D_{L^{n_{pc}-1}(p_{c})}L''' = D_{L^{n_{pc}-1}(p_{c}) }L$. We also have the equality : $D_{p_{c}}L \cdot ... \cdot  D_{L^{n_{pc}-1}(p_{c}) }L  = a^{n_{p_{c}}} \cdot I_{2} $ because the period $n_{p_{c}}$ is a multiple of the order of $A$ (see Proposition \ref{r41}). Then we have : 
 $$D_{p_{c}}(L''')^{n_{p_{c}}} =  a^{n_{p_{c}}} \cdot \Big( I_{2}+ \varepsilon_{3} \cdot  \gamma_{2}(p_{c}) \cdot \begin{pmatrix} \kappa_{1} & \kappa_{2} \\  \kappa_{3} & \kappa_{4} \end{pmatrix} \cdot (D_{p_{c}}L)^{-1} \Big)$$ Let us denote by $M$ the matrix whose two columns are $D\Pi_{\mathfrak{pc}}(w_{1})$ and $D\Pi_{\mathfrak{pc}}(w_{2})$. We choose :  $$\begin{pmatrix} \kappa_{1} & \kappa_{2} \\  \kappa_{3} & \kappa_{4} \end{pmatrix} = \frac{1}{ \gamma_{2}(p_{c})} \cdot  M \cdot  \begin{pmatrix} 1 & 0 \\  0 & 0 \end{pmatrix} \cdot M^{-1} \cdot (D_{p_{c}}L)$$ Then : 
  $$D_{p_{c}}(L''')^{n_{p_{c}}} =  a^{n_{p_{c}}} \cdot M \cdot \begin{pmatrix} 1+\epsilon_{3} & 0 \\  0 & 1 \end{pmatrix} \cdot M^{-1}$$
  This equality implies that items 3,4 and 5 are satisfied and this ends the proof of the proposition.  \qedhere

 \end{proof}  We are now able to prove Proposition \ref{r97}.
\begin{proof}[Proof of Proposition \ref{r97} ]
We consider the holomorphic family of holomorphic maps $(L'''_{\varepsilon_{1},\varepsilon_{2},\varepsilon_{3}})_{(\varepsilon_{1},\varepsilon_{2},\varepsilon_{3}) \in \mathbb{D}^{3}}$. According to Lemma \ref{r93}, every $L'''_{\varepsilon_{1},\varepsilon_{2},\varepsilon_{3}}$ is diagonalizable and admits $w_{1}$ and $w_{2}$ as eigenvectors. Then according to Lemma \ref{r99}, we can take some uniform open set set $\mathbb{B}_{lin} \subset \mathbb{P}^{2}(\mathbb{C})$, some ball $B_{lin}$, such that there exists $\varphi_{L'''}$ defined on $\mathbb{B}_{lin}$ with values in $B_{lin} \subset \mathbb{C}^{2} $ such that $L'''$ is linearizable by $\varphi_{L'''}: \mathbb{B}_{lin} \mapsto B_{lin}$. Moreover $\varphi_{L'''}$ varies continuously with $L'''$. We denote by $\mathfrak{B}_{lin} $ some ball in $\Pi^{-1}(\mathbb{B}_{lin}) \subset \mathbb{T}$.

\begin{lem}
Let $\Gamma'$ be the diameter of $\mathfrak{B}_{lin}$ of direction $w_{1}$. Then there exists $n_{0}$ such that $\bigcup_{1 \le n \le n_{0}} \mathcal{L}^{n}(\Gamma')$ contains a $(0,w_{1})$-quasi-diameter of $\mathfrak{B}$.
\end{lem}
\begin{proof}
$\bigcup_{1 \le n \le + \infty} \mathcal{L}^{n}(\Gamma')$ is dense in $\mathbb{T}$ by the choice of $w_{1}$. Then there exists $n_{0}$ such that $\bigcup_{1 \le n \le n_{0}} \mathcal{L}^{n}(\Gamma')$ contains a $(0,w_{1})$-quasi-diameter of $\mathfrak{B}$.  \qedhere

\end{proof} From Lemma \ref{rlin}, we know that $\varphi_{L} \circ \Pi$ is linear. Rewriting this result in $\mathbb{P}^{2}(\mathbb{C})$ we have : 
\begin{co5}
Let $\Gamma''$ be the diameter of $B_{lin}$ of direction $(\varphi_{L} \circ \Pi)(w_{1})$. Then there exists $n_{0}$ such that $ \Pi^{-1}(\bigcup_{1 \le n \le n_{0}} L^{n}(\varphi_{L}^{-1}(\Gamma'')) \cap \mathbb{B})$ contains a $(0,w_{1})$-quasi-diameter of $\mathfrak{B}$.
\end{co5}
 By continuity of $L''' \mapsto \varphi_{L'''}$ (see Lemma \ref{r99}), we have the following perturbation result : 
\begin{co5} \label{r95}
There exists $\theta'>0$, some neighborhood $\mathcal{W}_{2}(L)$ of $L$ in $\mathrm{Hol}_{d'}$ and an integer $n_{0}$ such that for every $(\theta',(\varphi_{L} \circ \Pi)(w_{1}))$-quasi-diameter $\Gamma''$ of $B_{lin}$, for every $L''' \in \mathcal{W}_{2}(L)$, we have that $\Pi^{-1}(\bigcup_{1 \le n \le n_{0}} (L''')^{n} (  \varphi_{L'''}^{-1}(\Gamma'')) \cap \mathbb{B})$ contains a $(\theta,w_{1})$-quasi-diameter of $\mathfrak{B}$.
\end{co5} Remind that $w_{2}$ is not tangent to $\Pi^{-1}(\text{PCrit}(L))$ at $\mathfrak{pc}$. We can take a neighborhood $\mathcal{W}_{3}(L)$ of $L$ such that every map in $\mathcal{W}_{3}(L)$ for which $\mathfrak{pc}$ is in the postcritical set still satisfies this condition. 
We consider $\mathcal{W}(L) = \mathcal{W}_{1}(L) \cap \mathcal{W}_{2}(L) \cap \mathcal{W}_{3}(L)$. Since the conclusions 1,2,3 and 4 of Lemma \ref{r93} are satisfied, according to Lemma \ref{r98}, there exists a disk $\tilde{\Gamma}$ included in the postcritical set of $L'''$ such that $\Pi^{-1}(\tilde{\Gamma})$ contains a $(\theta',w_{1})$-quasi-diameter of $\mathfrak{B}_{lin}$ (remind $\theta'$ was defined in Lemma \ref{r95}). According to Lemma \ref{r95},  $\bigcup_{1 \le n \le n_{0}} (\mathcal{L}''')^{n}(\Pi^{-1}(\tilde{\Gamma}))$ contains a $(\theta,w_{1})$-quasi-diameter of $\mathfrak{B}$ so the conclusion follows. \end{proof}

\section{Proof of the main result}
 We consider the perturbations $L'''$ in $\mathcal{W}(L)$ as in the previous subsection and such that $L'''=L'''_{\varepsilon_{1},\varepsilon_{2},\varepsilon_{3}} \in \mathcal{X}_{\varepsilon_{1}}$ (the neighborhood $\mathcal{X}_{\varepsilon_{1}}$ of $L_{\varepsilon_{1}}$ was introduced in Proposition \ref{r71}, all maps in $\mathcal{X}_{\varepsilon_{1}}$ have a correcting IFS). Let us consider the union of all the sets $\mathcal{G}'''_{j}(\mathfrak{B}) \subset \mathfrak{B}$ for $1 \le j \le q$ (remind that $\mathcal{G}_{j}'''  $ was defined in Proposition \ref{r71}). Reducing $\mathcal{W}(L)$ if necessary, by continuity it contains a grid of balls $G^{1}= (u^{1},o^{1},n_{G},r^{1})$ with $r^{1} \ge \frac{\iota \rho}{2}$ (see 4.3.20).

 \begin{pr}
 There exists an integer $d$ (depending only from $\mathbb{T}$) such that for every Latt\`es map $L$ inducing an affine map on $\mathbb{T}$ of linear part $aA$, every map $L'''$ as given in Proposition \ref{r97} is such that :

 \begin{enumerate} \item $\bigcup_{ 1 \le j \le q} \mathcal{G}'''_{j} (\mathfrak{B})$ contains a grid of balls $G^{1}= (u^{1},o^{1},n_{G},r^{1})$ with $q = (2n_{G}+1)^{4}$ such that each $\mathcal{G}'''_{j}  (\mathfrak{B})$ contains a ball of $G^{1}$ \item the contraction factor of the IFS $(\mathcal{G}'''_{1},...,\mathcal{G}'''_{q})$ is $|a| \ge 2$ \item 
there exist $(n+1)$ balls $\mathfrak{B}_{0},\mathfrak{B}_{1},...,\mathfrak{B}_{n}  \subset \mathfrak{B}$ of relative size larger than $\nu$, such that the $\frac{3}{4}$-parts of $\mathfrak{B}_{0},\mathfrak{B}_{1},...,\mathfrak{B}_{n}$ are included in the hull of $G^{1}$, and satisfying the following property :  for $1 \le j \le q$ such that $\mathcal{G}'''_{j}(\mathfrak{B}) \subset \mathfrak{B}_{p}$, $\mathcal{G}'''_{j} = \frac{1}{a}(A+\tilde{h}_{j})$ is quasi-linear of type $(x,p)$ with $x<x(u^{1})$ . Moreover, $\bigcup_{ 1 \le j \le q}  \mathcal{G}'''_{j}  (\mathfrak{B}_{p})$ contains a grid of balls $\Gamma^{1}_{p}= (u^{1},o^{1}_{p},n_{G},s^{1})$ for each $0 \le p \le n$ with $s^{1} \ge \nu \cdot r^{1}$

\item $n_{G}>\frac{10}{\nu} \cdot N(\frac{\nu \cdot r^{1}}{10})$ \item $|a| \cdot R \cdot \max_{1 \le j \le q} (|| \mathcal{G}'''_{j} ||_{C^{2}}   ) < \frac{\nu \cdot r^{1}}{100}$ \item there exists a $(\theta,w)$-quasi-diameter of $\mathfrak{B}$ inside  $\Pi^{-1}(    \mathrm{PCrit}(L''')   )$ \end{enumerate}

\end{pr}
\begin{proof}
The first item was stated before the proposition. The second one comes from 4.2.12 and the fourth one from 4.3.20. The fifth one can be obtained from a reduction of $\mathcal{W}(L)$ if necessary. The last one is a consequence of Proposition \ref{r97}. We show the third item. The existence of the balls $\mathfrak{B}_{p}$ of relative size $\nu$ is a consequence of Proposition \ref{r71}. The inclusions $\mathfrak{B}_{0},\mathfrak{B}_{1},...,\mathfrak{B}_{n}  \subset \mathfrak{B}$ and the inequality on $n_{G}$ ensure that there are sufficiently many $\mathcal{G}'''_{j} (\mathfrak{B})$ so that the $\frac{3}{4}$-parts of $\mathfrak{B}_{0},\mathfrak{B}_{1},...,\mathfrak{B}_{n}$ are included in the hull of $G^{1}$. Let us now consider the $(n+1)$ sets $ \bigcup_{j \le q} \mathcal{G}'''_{j}(\mathfrak{B}_{p}) \subset \mathfrak{B}$ for $0 \le p \le n$. Reducing $\mathcal{W}(L)$ a last time if necessary, by continuity each of them contains a grid of balls $\Gamma^{1}_{p}= (u^{1},o^{1}_{p},n_{G},s^{1})$ with $s^{1} \ge \nu \cdot r^{1}$. The property stated in item 3 is also a consequence of Proposition \ref{r71}.  
\end{proof}
The intersection $\bigcap_{j \ge 1} \mathcal{G}_{j}(\Pi(\mathfrak{B}))$ is in the Julia set of $L$. Since $\bigcap_{j \ge 1} \mathcal{G}_{j}(\Pi(\mathfrak{B}))$ is a basic repeller, it is a consequence of Lemma 2.3 of \cite{dm} that $\bigcap_{j \ge 1} \mathcal{G}'''_{j}(\Pi(\mathfrak{B}))$ is in the Julia set of $L'''$ for sufficiently small perturbations $L'''$ of $L$.  According to Proposition \ref{r99999} (beware that the maps $\mathcal{G}'''_{j}$ in our case correspond to the maps $\mathcal{G}_{j}$ of the proposition), we can conclude this gives us persistent intersections between the Julia set and the postcritical set. This is true for every $L''' $ defined as before and we know that $L$ is accumulated by such maps inside $\text{Hol}_{d'}$. By \cite{du9} (see Proposition 2.5 of \cite{dm} for a result in our case), we know that persistent intersections between the postcritical set and a hyperbolic repeller inside the Julia set imply the presence of open sets inside the bifurcation locus. Since they are only finitely many $A \in G_{\text{Latt\`es}}$ for a given torus $\mathbb{T}$, $d$ is well defined. This proves the final result. \newline \newline
  \bibliographystyle{plain} \bibliography{bibi14}  \end{document}